\documentclass[12pt]{amsart}
\usepackage{enumerate}
\usepackage{amsmath, amscd, amsfonts, amsthm, amssymb, latexsym, 
	comment, stmaryrd, graphicx}
\usepackage[a4paper]{geometry}
\usepackage{xcolor}
\usepackage[all]{xy}
\usepackage[
colorlinks,
backref,
]{hyperref}

\theoremstyle{plain}
\newtheorem{theorem}{Theorem}[section]
\newtheorem{proposition}[theorem]{Proposition}
\newtheorem{lemma}[theorem]{Lemma}
\newtheorem{conjecture}[theorem]{Conjecture}

\theoremstyle{definition}
\newtheorem{example}[theorem]{Example}
\newtheorem{definition}[theorem]{Definition}
\newtheorem{remark}[theorem]{Remark}

\newcommand{\ord}{{\rm ord}}
\newcommand{\genus}{{\rm genus}}
\newcommand{\Gal}{{\rm Gal}}
\newcommand{\ZZ}{\mathbb{Z}}
\newcommand{\Li}{\mathrm{Li}}

\newcommand\FF{\mathbb{F}}
\newcommand\CC{\mathbb{C}}
\newcommand\QQ{\mathbb{Q}}
\newcommand\AAA{\mathbb{A}}

\newcommand\disc{{\rm disc}}

\begin{document}

\title[Chebotarev density theorem in short intervals]{Chebotarev density theorem in short intervals for extensions of $\FF_q(T)$}
\author{Lior Bary-Soroker}
\address{Raymond and Beverly Sackler School of Mathematical Sciences, Tel Aviv University, Tel Aviv 69978, Israel}
\email{barylior@post.tau.ac.il}
\thanks{LBS  was partially supported by a grant of the Israel Science Foundation.
	Part of the work was done while LBS was a member of Simons CRM Scholar-in-Residence Program.}
\author{Ofir Gorodetsky}
\address{Raymond and Beverly Sackler School of Mathematical Sciences, Tel Aviv University, Tel Aviv 69978, Israel}
\email{ofir.goro@gmail.com}
\thanks{The research of OG was supported by the European Research Council under the European Union's Seventh Framework Programme (FP7/2007-2013) / ERC grant agreement n$^{\text{o}}$ 320755.}

\author{Taelin Karidi}
\address{Department of Mathematics, Caltech, Pasadena, CA 91125, USA}
\email{tkaridi@caltech.edu}

\author{Will Sawin}
\address{Department of Mathematics, Columbia University, New York, NY 10027, USA}
\email{sawin@math.columbia.edu}
\thanks{This research was partially conducted during the period WS served as a Clay Research Fellow, and partially conducted during the period he was supported by  Dr. Max R\"{o}ssler, the Walter Haefner Foundation and the ETH Zurich Foundation.}


\begin{abstract}
An old open problem in number theory is whether Chebotarev density theorem holds in short intervals. More precisely, given a Galois extension $E$ of $\QQ$ with Galois group $G$, a conjugacy class $C$ in $G$ and an $1\geq \varepsilon>0$, one wants to compute the asymptotic of the number of primes $x\leq p\leq x+x^{\varepsilon}$ with Frobenius conjugacy class in $E$ equal to $C$. The level of difficulty grows as $\varepsilon$ becomes smaller. Assuming the Generalized Riemann Hypothesis, one can merely reach the regime $1\geq\varepsilon>1/2$. We establish a function field analogue of Chebotarev theorem in short intervals for any $\varepsilon>0$. Our result is valid in the limit when the size of the finite field tends to $\infty$ and when the extension is tamely ramified at infinity. The methods are based on a higher dimensional explicit Chebotarev theorem, and applied in a much more general setting of arithmetic functions, which we name $G$-factorization arithmetic functions.
\end{abstract}

\maketitle
\section{Introduction}
The goal of this paper is to provide support to an open problem in the distribution of primes with a given Frobenius conjugacy class. We do this by resolving a function field version of the problem. We start by introducing the problem in number fields, and then we present our results. 
\subsection{The Chebotarev Density Theorem in short intervals}
One of the main theorems in algebraic number theory is the Chebotarev Density Theorem about the distribution of Frobenius conjugacy classes in Galois extensions of global fields. To keep the presentation as simple as possible, we fix the base field to be $\QQ$.
Let $E$ be a finite Galois extension of $\QQ$ with Galois group $G=\Gal(E/\QQ)$ and with ring of integers $\mathcal{O}_E$. For a prime number $p$, we define
\[
\left( \frac{E/\QQ}{ p} \right)\subseteq G,
\]
to be the set of all $\sigma \in G$ for which there exists a prime $\mathfrak{P}$ of $E$ lying above $ p$ such that
\[
\sigma (x) \equiv x^{p} \mod \mathfrak{P},
\]
for all $x\in \mathcal{O}_E$.  If $p$ is unramified in $E$, then $\left( \frac{E/\QQ}{ p} \right)$ is called the \emph{Frobenius} at $p$ and it is a conjugacy class in $G$.

The Chebotarev Density Theorem says that as $p$ varies, the Frobenius equidistributes   in the set of conjugacy classes (with the obvious weights). More precisely, let
\[
\pi(x) = \# \{ p \leq x : p \mbox{ prime number}\}
\]
be the prime counting function. By the Prime Number Theorem, we know that $\Li(x) = \int_{2}^x \frac{dt}{\log t} \sim \frac{x}{\log x}$ well approximates $\pi(x)$; that is to say, for any $A>1$ we have
\[
\pi(x) = \Li(x) + O_A(x/(\log x)^A), \qquad x\to\infty.
\]
For a  conjugacy class $C \subseteq G$, let
\begin{equation*}
\pi_C(x;E)= \#\left\{ p \le x : p \mbox{ prime number} \mbox{ and } \left( \frac{E/\QQ}{p} \right)=C\right\}
\end{equation*}
be the function that counts primes with Frobenius equals to $C$. 
The Chebotarev Density Theorem \cite[Theorem~2.2, Chapter~I]{Serre1998} says that
\begin{equation}\label{eq:Cheb}
\pi_C(x;E) \sim \frac{|C|}{|G|} \Li(x), \qquad x\to \infty.
\end{equation}
This theorem is a vast generalization of the Prime Number Theorem for  arithmetic progressions which follows from \eqref{eq:Cheb} applied to cyclotomic fields.

It is both natural and important for applications to consider the Chebotarev Density Theorem in short intervals. Balog and Ono \cite{balog2001} studied the non-vanishing of Fourier coefficients of modular forms in short intervals. For this application they prove that
\begin{equation}\label{eq:SI}
\pi_{C}(x+y;E) -\pi_{C}(x;E) \sim \frac{|C|}{|G|} \frac{y}{\log x}, \qquad x\to\infty,
\end{equation}
for $x^{1-1/c(E)+\varepsilon} \leq y\leq x$, and where $c(E)>0$ is a constant depending only on $E$ (and in fact only on $[E:\QQ]$). Thorner \cite[Corollary~1.1]{thorner2016} improves the range of $y$ for which \eqref{eq:SI} holds true.

 Naively, we expect that \eqref{eq:SI} holds for any $y=y(x)\leq x$ that grows
`sufficiently fast'.
From \eqref{eq:Cheb}, it follows that the average gap between primes with $\left(\frac{E/\QQ}{p}\right) = C$ is $\frac{|G|}{|C|}\log x$. Thus it makes sense to only consider $y$-s satisfying $\lim_{x\to \infty} \frac{y}{\log x} = \infty$.
The Maier phenomenon \cite{maier1985} about primes tells us that  \eqref{eq:SI} fails unless $y \gg (\log x)^A$ for all $A> 1$. 
A folklore conjecture says that for any fixed $\varepsilon>0$ and $y=x^{\varepsilon}$  
the asymptotic formula \eqref{eq:SI} holds true:
\begin{conjecture}\label{conj:Cheb_SI}
Let $E/\QQ$ be a Galois extension with Galois group $G$,  $1\geq \varepsilon>0$, and  $C\subseteq G$  a conjugacy class. Then
\[
\pi_C(x+x^{\varepsilon};E) - \pi_C(x;E) \sim \frac{|C|}{|G|} \frac{x^{\varepsilon}}{\log x}, \qquad x\to \infty.
\]
\end{conjecture}

When $E=\QQ$, Conjecture~\ref{conj:Cheb_SI} reduces to primes in short intervals, and we refer the reader to the excellent survey \cite{sound2007} for further reading on this case.

One approach for Conjecture~\ref{conj:Cheb_SI} is to study the error term in Chebotarev Density Theorem.
Let
\[
\Delta_{E;C}(x)= \pi_C(x,E) - \frac{|C|}{|G|}\Li(x)
\]
and let $d_E$ be the absolute value of the discriminant of $E$. Under the Riemann Hypothesis for the Dedekind zeta function  $\zeta_E$ of  $E$,
 Lagarias and Odlyzko \cite{lagarias1975} gave the bound
\begin{equation}\label{eq:DeltaECBound}
\Delta_{E;C}(x) = O(|C| x^{1/2}( \log x + \frac{\log d_E}{|G|}) ),
\end{equation}
where the implied constant is effective and absolute. We borrow the above formulation from \cite[Theorem.~4]{serre1981}. See \cite[Cor.~1]{grenie2019} for a calculation of the implied constants and \cite[Cor.~3.7]{rammurty1988} for an improved  dependence on $|C|$.

From \eqref{eq:DeltaECBound}, in particular conditionally on the Riemann Hypothesis for $\zeta_E$, one immediately gets Conjecture~\ref{conj:Cheb_SI} for any $\varepsilon>1/2$. As discussed above, there are  unconditional results. However, the case $\varepsilon \leq 1/2$ falls beyond the Generalized Riemann Hypothesis.

\subsection{The Chebotarev Density Theorem in function fields}\label{sec:ChebFunc}
The function field Chebotarev Density Theorem has a long history, starting with
Reichardt \cite{reichardt1936} who first established it. Lang \cite{lang1956} gave a square-root cancellation, based on the Riemann Hypothesis for curves over finite fields, and explicit estimates  were given by Cohen and Odoni in the appendix to  \cite{cohen1977} and by Halter-Koch \cite[Satz~2]{halterkoch1991}. Fried and Jarden \cite[Proposition~6.4.8]{fried2005} and Murty and Scherk \cite{kumarmurty1994} gave explicit bounds on the error term. 

However, unlike the number field case, there are two obstructions in the Chebotarev Density Theorem. One obstruction comes from the arithmetic part of the Frobenius and the other appears when considering short intervals. For a more concise presentation of the obstruction we introduce some notation.

Let $q$ be a power of a prime number $p$, let $\FF_q$ be the finite field of $q$ elements, and let $\FF_q(T)$ be the field of rational functions over $\FF_q$. We define $\mathcal{P}_{n,q}$ as the set of primes of $\FF_q(T)$ of degree $n$. If $n>1$, we identify $\mathcal{P}_{n,q}$ with the set of monic irreducible polynomials in the ring of polynomials $\FF_q[T]$, and we identify $\mathcal{P}_{1,q}$ with the degree-$1$ monic polynomials and $1/T$ `the infinite prime'. The prime polynomial theorem says that
\[
\pi_q(n) = \# \mathcal{P}_{n,q} = \frac{q^n}{n}(1 +O(q^{-n/2})),
\]
and so we use $\frac{q^n}{n}$ as an estimate for $\pi_q(n)$.

Given a Galois extension $E/\FF_q(T)$ with Galois group $G=\Gal(E/\FF_q(T))$, for each $P\in \mathcal{P}_{n,q}$
we define the \emph{Frobenius at $P$},
\begin{equation}\label{eq:FrobatP}
\left( \frac{E/\FF_q(T)}{ P} \right)\subseteq G,
\end{equation}
as in the number field setting: it is the set of $\sigma\in G$ for which there exists a prime $\mathfrak{P}$ of $E$ lying above $P$ such that
\[
\sigma(x) \equiv x^{|P|} \mod \mathfrak{P},
\]
for all $x\in E$ which are integral at $\mathfrak{P}$ and where $|P|=q^{n}$. As before, if $P$ is unramified in $E$, then $\left( \frac{E/\FF_q(T)}{ P} \right)$ is a conjugacy class in $G$.
Given a conjugacy class $C \subseteq G$, we set
\begin{equation*}
\pi_{C;q}(n;E) =  \# \left\{ P \in \mathcal{P}_{n,q}:  \left(\frac{E/\FF_q(T)}{P}\right) = C \right\},
\end{equation*}
the function that counts primes with Frobenius $C$.

To describe the obstruction for a conjugacy class to be a Frobenius of a prime of degree $n$, we introduce the restriction map. Let $\FF_{q^{\nu}}$ be the field of scalars of $E$, that is, the algebraic closure of $\FF_q$ in $E$. Let $\phi\colon \FF_{q^{\nu}} \to \FF_{q^{\nu}}$, $\phi(x)=x^q$ be the generator of the cyclic group $G_0=\Gal(\FF_{q^{\nu}}/\FF_q)$. We have the restriction of automorphisms map $G\twoheadrightarrow G_0$, which is surjective. Since $G_0$ is abelian, if $C\subseteq G$ is a conjugacy class, then all $\sigma\in C$ map to the same power $\phi_C$ of $\phi$.
Then the Chebotarev Density Theorem for function fields says that if $\phi_C = \phi^n$, then
\begin{equation}\label{eq:Cheb_FF}
\left|\pi_{C;q}(n;E) - \nu \frac{|C|}{|G|} \frac{q^n}{n}\right| \ll \nu \frac{|C|}{|G|} \max\{  \mathrm{genus}(E) , \frac{|G|}{\nu}\}\frac{q^{n/2}}{n}
\end{equation}
and otherwise $\pi_{C;q}(n;E)=0$. The implied constant is absolute.

Next we turn to short intervals. 
Following Keating and Rudnick \cite[\S2.1]{keating2014}, we define a short interval around a polynomial $f$ of degree $n$ with parameter $0\le m<n$ to be
\[
I(f,m) = \{ f+g : \deg g\leq m\}.
\]
The size of the interval is 
\[
\# I(f,m)=q^{m+1}.
\]
To compare with the number field interval $\{ x\leq n\leq x+x^\varepsilon\}$, we see that $x$ corresponds to $|f|=q^n$ and $x^\varepsilon$ corresponds to $q^{m+1}$, so $\varepsilon=\frac{m+1}{n}$. Having the analogy with number fields in mind, one would naively expect that \eqref{eq:Cheb_FF} implies a Chebotarev  Density Theorem for the short interval $I(f,m)$ whenever $m+1> n/2$ (i.e., $\varepsilon > 1/2$).
However, there seems to be no direct such implication.
Letting
\[
\pi_{C;q}(I(f,m);E) = \# \left\{ P \in P_{n,q}\cap I(f,m)  : \left(\frac{E/\FF_q(T)}{P}\right)=C \right\} ,
\]
then unlike in the number field case, we cannot express $\pi_{C;q}(I(F,m);E)$ as the difference of values of  $\pi_{C;q}(n;E)$ in order to utilize the error term \eqref{eq:Cheb_FF}.

In fact, there is an obstruction to Chebotarev in short intervals coming from the fact that $E$ is not necessarily linearly disjoint from the cyclotomic field $L_{n-m-1}$ associated to a power of the infinite prime (see \cite[Chapter~12]{rosen2002}). Thus one needs to modify the asymptotic formula according to the intersection of $E$ and $L_{n-m-1}$. 
Applying \eqref{eq:Cheb_FF} to the compositum of $EL_{n-m-1}$ would yield a Chebotarev in short intervals for $\varepsilon >1/2$. We note that the extensions $L_{n-m-1}$ are wildly ramified at the infinite prime.

Our main result is a Chebotarev Density Theorem for short intervals with any $\varepsilon>0$ for extensions that are tamely ramified at the infinite prime. Thus the result goes beyond the Riemann Hypothesis. For simplicity of presentation we consider only geometric extensions. We indicate at the end of the paper how to handle non-geometric extensions. 

\begin{theorem}\label{thm:chebotarevshort}
For every $B>0$ there exists a constant $M_B$ satisfying the following property. Let $q$ be a prime power. Let $n>m \ge 2$ if $q$ is odd and $n> m \ge 3$ otherwise. Let $G$ be a finite group and let $E/\FF_q(T)$ be a geometric $G$-extension. Assume that the infinite prime is tamely ramified in the fixed field $ E^{ab}$ in $E$ of the commutator of $G$. Further assume that $\genus(E),n,|G|\leq B$. Let $f\in \FF_q[T]$ be monic of degree $n$. Then
\[
\left|\frac{1}{q^{m+1}}\pi_{C;q}(I(f,m);E) - 
\frac{|C|}{|G|}\frac{1}{n}\right|\leq M_B q^{-1/2}.
\]
\end{theorem}

It follows in particular that for any $\varepsilon>0$ we have 
\begin{equation}\label{eq:doublelimit}
\lim_{n\to \infty} \lim_{q\to \infty} \max_{f,E}\left|\frac{1}{q^{m+1}}\pi_{C;q}(I(f,m);E) -  \frac{|C|}{|G|}\frac{1}{n}\right| = 0,
\end{equation}
where $E$ runs over all $G$-Galois extensions of $\FF_q(T)$  of bounded genus that are tamely ramified at infinity and $f\in \FF_q[T]$ runs over all monic polynomials of degree $n$. Hence we have proved a version of Conjecture~\ref{conj:Cheb_SI} in the function field setting. 

It would be desirable to change the order of the limits in \eqref{eq:doublelimit}. As explained above, for $\varepsilon>1/2$ this follows from  the Riemann Hypothesis for curves. For  $\varepsilon\leq \frac12$, it is open and we know of no approach to attack it. 
A yet more challenging task is to fix $q$ and take $n\to \infty$, and also here the problem is open, and we know of no approach to attack it. 

Our method gives more general results, and may be applied for instance to problems about norms. In Theorem \ref{thm:norms} we count, in the large-$q$ limit, how many polynomials $g \in I(f,m)$ satisfy $(g)=\mathrm{Norm}_{E/\FF_q(T)}I$ for some ideal $I$ in $\mathcal{O}_E$. Our most general result is given in Theorem~\ref{thm:maintechnical}, for which the terminology of \S\ref{sec:GFact}--\ref{sec:wreath} is needed.

It would be interesting to generalize our results to a function field of a general curve in place of $\FF_q(T)$.

\section{Methods}
We outline our approach when $E$ is a geometric extension of  $\FF_q$, which, under the notation used in \eqref{eq:Cheb_FF}, means that $\nu=1$.  We introduce a general notion of $G$-factorization arithmetic functions (Definition~\ref{def:Gfact}), which are arithmetic functions on $\FF_q[T]$, whose value on a polynomial $f(T)$ depends only on the Frobenius at the prime factors of $f(T)$. These functions are closely related to Serre's Frobenian functions \cite{serre1975} and to the extensions by Odoni \cite{odoni1981,odoni1982} and Coleman \cite{coleman2001}.

Given a short interval, we relate such an arithmetic function $\psi$ to a class function $\psi'$ on a \emph{subgroup} of the wreath product $G \wr S_n$ using a  higher dimensional function field Chebotarev Density Theorem. The main property of this association is that the expected value of $\psi$ on the short  interval is asymptotically equal to the average of $\psi'$ on the subgroup, as $q \to \infty$ (Theorem~\ref{thm:maintechnical}). 
The main technical part of the work is to compute the subgroup: it equals to the wreath product $G \wr S_n$ itself. 

Applying the above to the indicator function of primes with Frobenius $C$ (Example~\ref{example:prime}) reduces Theorem~\ref{thm:chebotarevshort} to either a combinatorial computation in $G\wr S_n$  or the classical Chebotarev Density Theorem.

Finally, for the subgroup computation, we take an algebraic approach, using elementary group theory and Artin-Schreier and Kummer theories. Our methods are in the spirit of the works  \cite{cohen1980,bank2015a,bank2018} which assume $\genus(E)=0$ and $G$ cyclic. 

\section{\texorpdfstring{$G$}{G}-factorization arithmetic functions}\label{sec:GFact}
For a finite group $G$ we consider the space
\[
\hat{\Omega}_G = \{ \sigma  I : \sigma\in G,\ I\leq G\}
\]
of all cosets of subgroups. The group $G$ acts on $\hat{\Omega}_G$ by conjugation and we write
\[
\Omega_G = \hat{\Omega}_G/G
\]
for the set of conjugacy classes of cosets of subgroups.
If $I=1$ is the trivial subgroup, we identify $\sigma  I\in \hat{\Omega_G}$ with $\sigma$. So the image of $\sigma I$ in $\Omega_G$ is the conjugacy class $C=\{\tau^{-1} \sigma \tau : \tau\in G\}$ of $\sigma$.

We want to encode the combinatorial  data of degrees, multiplicities, and the Frobenius  at the prime factors of a polynomial.
A \textbf{$G$-factorization type} is a function
\begin{equation*}
\lambda\colon \mathbb{N} \times \mathbb{N} \times \Omega_G\to \ZZ_{\ge 0}
\end{equation*}
with finite support. We define $\Lambda=\Lambda_G$ to be  the set of all $G$-factorization types. For $\lambda \in \Lambda$ we let
\[
\begin{split}
\deg(\lambda) &= \sum_{d,e,\omega}\lambda(d,e,\omega) de,\\
\end{split}
\]
where the sum runs over $d,e\in \mathbb{N}$ and $\omega\in \Omega_G$. For a monic polynomial $f\in \FF_q[T]$  with prime factorization $f=P_1^{e_1} \cdots P_r^{e_r}$ and for a $G$-Galois extension $E/\FF_q(T)$  we define
\[
\lambda_{f;E/\FF_q(T)}(d,e,\omega) = \# \left\{ i : \deg P_i=d,\ e_i=e,\ \left(\frac{E/\FF_q(T)}{P_i}\right)=\omega\right\}.
\]
When there is no risk of confusion we simplify the notation and write $\lambda_f$ for $\lambda_{f;E/\FF_q(T)}$.
Obviously, we have that $\deg(f) = \deg(\lambda_f)$.
\begin{definition}\label{def:Gfact}
A \textbf{$G$-factorization arithmetic function} is a function on $G$-factorization types. We denote by 
\[
\Lambda^*  = \{\psi\colon \Lambda \to \CC\}
\]
the space of $G$-factorization arithmetic functions. 

Given a $G$-Galois extension $E/\FF_q(T)$, each  $\psi\in \Lambda^*$ induces an arithmetic function $\psi_{E/\FF_q(T)}$ on $\FF_q[T]$ by setting
\[
\psi_{E/\FF_q(T)}(f) = \psi(\lambda_{f;E/\FF_q(T)}),
\]
for monic $f\in \FF_q[T]$. By abuse of notation, $\psi_{E/\FF_q(T)}$ is also called $G$-factorization arithmetic function.
\end{definition}
Definition~\ref{def:Gfact} vastly extends some families of arithmetic functions -- see  \cite{rodgers2018,LBSFehm} for similar definitions in the cases $E=\FF_q(T)$ ($G=\{e\}$) and $E=\FF_q(\sqrt{-T})$ ($G=\ZZ/2\ZZ$).

The following example of a $G$-factorization arithmetic function is crucial for our main result.
\begin{example}\label{example:prime}
Fix a conjugacy class $C \subseteq G$. Consider the $G$-factorization arithmetic function
\[
1_{C}(\lambda) =
\begin{cases}
1, & \mbox{if }\lambda(d,e,\omega) > 0 \Rightarrow \omega = C \mbox{ and }d = \deg \lambda, \\
0, & \mbox{otherwise.}
\end{cases}
\]
For any $G$-Galois extension $E/\FF_q(T)$ and monic $f\in \FF_q[T]$ we have
\begin{equation*}
1_{C,E/\FF_q(T)}(f) = \begin{cases}
1, & \mbox{if $f$ is irreducible and } \left(\frac{E/\FF_q(T)}{f}\right)=C,\\
0,&  \mbox{otherwise}.
\end{cases}
\end{equation*}
\end{example}

\section{\texorpdfstring{$G$}{G}-factorization arithmetic functions on wreath products}
\label{sec:wreath}

Recall the construction of the \textbf{permutational wreath product}:
Let $S_n$ be the symmetric group on $X=\{1,2,\ldots,n\}$ (with left action $(\sigma,x)\mapsto \sigma.x$), let $G$ be a finite group,
and let
\[
 G^{X}:= \{ \xi\colon X \to G\}
\]
be the
group of functions from $X$ to $G$ with pointwise multiplication.
Then $S_n$ acts (from the right) on $G^{X}$ by
\[
\xi^\sigma(x) = \xi(\sigma.x), \qquad \sigma\in S_n, \ x\in X.
\]
The corresponding semidirect product
\[
G\wr S_n := G^X\rtimes S_n
\]
is  the  wreath product of $G$ and $S_n$. For the reader's convenience we recall that the multiplication is given by 
\begin{equation*}
(\xi_1,\sigma_1)(\xi_2,\sigma_2) = (\xi_1\xi_2^{\sigma_1^{-1}}, \sigma_1 \sigma_2), \qquad \xi_1, \xi_2 \in G^X, \ \sigma_1, \sigma_2 \in S_n.
\end{equation*}
The imprimitive action of $G\wr S_n$ on the set $G\times X$, given explicitly by
\begin{equation}\label{eq:imprimitiveaction}
(\xi,\sigma).(g,x) = (\xi(\sigma.x)g,\sigma.x), \qquad \xi\in G^X, \ \sigma\in S_n, \ g\in G,\ x\in X,
\end{equation}
makes $G\wr S_n$ into a transitive permutation group.

For $(\xi, \sigma)\in G\wr S_n$ we attach a $G$-factorization type: Let $\sigma = \sigma_1 \cdots \sigma_r$ be the factorization of $\sigma$ to disjoint cycles. We include the trivial cycles so that $\sum_{i=1}^r {\ord}(\sigma_i) = n$. For each $i=1,\ldots, r$, if we write $\sigma_i = (j_1 \ \cdots\ j_d)$, then we set $C_{(\xi,\sigma),\sigma_i}$ to be the conjugacy class in $G$ of the element
\[
\xi(j_d) \cdots \xi(j_1).
\] 
The conjugacy class $C_{(\xi,\sigma),\sigma_i}$ is well defined, since $\xi(j_{a}) \cdots \xi(j_1) \xi(j_d)\cdots \xi(j_{a+1})$ is conjugate to $\xi(j_d) \cdots \xi(j_1)$. Now we set 
\begin{equation}\label{eq_lambda}
\lambda_{(\xi,\sigma)}(d,e,\omega) =
\begin{cases}
0, & \mbox{if }e>1,\\
\# \{ i : {\ord}(\sigma_i) = d,\ C_{(\xi,\sigma),\sigma_i}  = w\}, & \mbox{if }e=1.
\end{cases}
\end{equation}
Any $\psi\in \Lambda^*$ induces a function $\psi_{G \wr S_n}\colon G \wr S_n \to \CC$ by
\[
\psi_{G\wr S_n}((\xi,\sigma)) = \psi(\lambda_{(\xi,\sigma)})
\]
and we refer to such functions on $G\wr S_n$ as $G$-factorization arithmetic functions as well. Below we show that the set of $G$-factorization arithmetic functions on $G\wr S_n$ actually coincides with the set of class functions.

\begin{example}\label{exa:1CGG}
Recall the $G$-factorization arithmetic function $1_C$ from Example~\ref{example:prime}. Then, for $(\xi,\sigma)\in G\wr S_n$ we have 
\[
1_C(\xi,\sigma ) = 
	\begin{cases}
		1,& \mbox{if $\sigma$ is an $n$-cycle and } C_{(\xi,\sigma),\sigma}=C,\\
		0,&\mbox{otherwise.}
		
	\end{cases}
\]
\end{example}

Next, we prove that conjugation in $G\wr S_n$ preserve the $G$-factorization type. Let $\tau\in S_n$ and identify it with $(1,\tau)\in G\wr S_n$. Then $\tau \sigma \tau^{-1} = \rho_1 \cdots \rho_r$ with $\rho_i=\tau\sigma_i\tau^{-1}$. If $\sigma_i = (j_1\ \cdots\ j_d)$, then $\rho_i = (\tau(j_1)\ \cdots \ \tau(j_d))$. Now, as
$\tau(\xi,\sigma)\tau^{-1} = (\xi^{\tau^{-1}} , \tau\sigma\tau^{-1})$
we have that 
\begin{equation}\label{eq:conj-S_n}
\xi^{\tau^{-1}}(\tau(j_d))\cdots \xi^{\tau^{-1}}(\tau(j_1)) = \xi(j_d)\cdots \xi(j_1) 
\end{equation}
and so $C_{(\xi,\sigma),\sigma_i} = C_{\tau(\xi,\sigma)\tau^{-1} ,\rho_i}$. We thus conclude that 
\[
\lambda_{(\xi,\rho)} = \lambda_{\tau (\xi,\rho) \tau^{-1}}. 
\]
Similarly, if $\eta\in G^X$ and we identify it with $(\eta,1)\in G\wr S_n$, then 
\begin{equation}
\label{eq-conj-eta}
\eta (\xi,\sigma) \eta^{-1} = (\eta\xi\eta^{-\sigma^{-1}},\sigma)
\end{equation} 
and we have 
\begin{equation*}
\begin{split}
&(\eta\xi\eta^{-\sigma^{-1}})(j_d) \cdots (\eta\xi\eta^{-\sigma^{-1}}(j_1)) \\
	&\qquad= \eta (j_d) \xi(j_d) \eta(j_{d-1})^{-1} \cdot \eta(j_{d-1})\cdots \eta(j_1)^{-1}\cdot \eta(j_1) \xi(j_1) \eta(j_d)^{-1} \\
	&\qquad=\eta(j_d) \xi(j_d)\cdots \xi(j_1) \eta(j_d)^{-1}.
\end{split}
\end{equation*}
Here we used that $\sigma_i^{-1} = (j_d\ \cdots \ j_1)$. In particular, $C_{(\xi,\sigma),\sigma_i} = C_{\eta(\xi,\sigma)\eta^{-1}, \sigma_i}$ and thus
\[
\lambda_{(\xi,\rho)} = \lambda_{\eta(\xi,\rho)\eta^{-1}}.
\]
We thus deduce that if $(\xi,\sigma)$ and $(\eta,\rho)$ are conjugate, then $\lambda_{(\xi,\rho)}=\lambda_{(\eta,\rho)}$.

The converse is also true. Indeed, $\lambda_{(\xi,\sigma)}=\lambda_{(\zeta,\rho)}$ implies that we have $r$ conjugacy classes $C_1, \ldots, C_r$ (possibly with repetitions) and factorization to disjoint cycles $\rho=\rho_1\cdots \rho_r$ and $\sigma = \sigma_1\cdots\sigma_r$ such that $\ord(\sigma_i)= \ord(\rho_i)$ and 
\[
C_{(\xi,\sigma),\sigma_i} = C_i  = C_{(\zeta,\rho),\rho_i}.
\]
Without loss of generality we may assume that $\sigma = \rho$ and $\sigma_i=\rho_i$ for all $i$ (indeed, conjugate by $\tau\in S_n$ such that $\tau\sigma_i\tau^{-1}=\rho_i$ for all $i$ and use \eqref{eq:conj-S_n}).
Thus, if $\sigma_i = (j_1 \ \cdots \ j_d)$, then 
\[
g\xi(j_d) \cdots \xi(j_1)g^{-1} = \zeta(j_d) \cdots \zeta(j_1) 
\]
for some $g \in G$. By \eqref{eq-conj-eta} it suffices to find $\eta\in G^X$ such that $\eta\xi \eta^{-\sigma^{-1}} = \zeta$. 
Defining 
\[
\eta(j_d) = g \quad \mbox{and} \quad \eta(j_{a}) = \zeta^{-1}(j_{a+1}) \eta(j_{a+1}) \xi(j_{a+1}), \ 1\leq a\leq d-1
\] 
on the orbits of $\sigma_i$, for each $\sigma_i$ gives the desired solution. 
We thus proved 

\begin{lemma}\label{Lem_lambda_conj}
	The elements $(\xi, \sigma)$, $(\zeta,\rho ) \in G \wr S_n$ are conjugate if and only if $\lambda_{(\xi,\sigma)} =\lambda_{(\zeta,\rho )}$. 
	
	In particular, every class function on $G\wr S_n$ may be realized as a $G$-factorization arithmetic function.
\end{lemma}

We prove the following general theorem which connects the averages of a $G$-factorization arithmetic function on a short interval to the average on the wreath product. 
A piece of notation is needed: for a non-empty finite set $X$ and a function $\psi$ on $X$ we denote the mean value by 
\[
\left<\psi(f)\right>_{f\in X} := \frac{1}{\# X} \sum_{f\in X} \psi(f).
\]
\begin{theorem}\label{thm:maintechnical}
For every $B>0$ there exists a constant $M_B$ satisfying the following property. Let $q$ be a prime power. Let $n>m\geq 2$ if $q$ is odd and $n>m \geq 3$ otherwise. Let $G$ be a finite group and let $E/\FF_q(T)$ be a geometric $G$-extension. Assume that the infinite prime is tamely ramified in the fixed field $E^{ab}$ in $E$ of the commutator of $G$. Let $f_0\in \FF_q[T]$ be monic of degree $n$, and $\psi\in \Lambda^*$. Assume that $\genus(E), n,|G| \leq B$. Then 
\[
\left|\left<\psi_{E/\FF_q(T)}(f)\right>_{\deg(f-f_0)\leq m} - \left<\psi_{G\wr S_n}(\tau)\right>_{\tau\in G\wr S_n} \right| \leq M_B q^{-1/2} \max_{\deg(\lambda)=n} |\psi(\lambda)|.
\]
\end{theorem}

We postpone the proof of Theorem~\ref{thm:maintechnical} to \S\ref{sec:pfthm:maintechnical}. 

\section{Applications of Theorem~\ref{thm:maintechnical}}\label{sec:apps}

From Theorem~\ref{thm:maintechnical} it follows immediately that in the large-$q$ limit, the average on a short interval is the same as on the `long interval' --- the set of all degree-$n$ monics 
\[
M_{n,q} = I(T^n,n-1).
\]

Moreover, Theorem~\ref{thm:maintechnical} reduces the computations of averages of arithmetic functions to combinatorics of group theory. 
This also works vice versa. 

\subsection{Proof of Theorem~\ref{thm:chebotarevshort}}\label{sec:proofoftheoremnorm}
We give two proofs to exemplify the ways to apply Theorem~\ref{thm:maintechnical}.
\medskip

First proof:
The assumptions allow us to apply Theorem~\ref{thm:maintechnical} with the $G$-factorization arithmetic function  $1_C$, and to get that the average on a short interval is the same as over a long interval. The latter is given by \eqref{eq:Cheb_FF}, as needed.
\medskip

Second proof:
The assumptions allow us to apply Theorem~\ref{thm:maintechnical} with the $G$-factorization arithmetic function  $1_C$, and to get that the average on a short interval is the same as on the wreath product. We compute the latter:
Using Example~\ref{exa:1CGG}, we find that $1_C(\xi,\sigma)\neq 0$ implies that $\sigma = (j_1\ \cdots \ j_n)$ is an $n$-cycle  and $\xi(j_n) \cdots \xi(j_1)\in C$. So we may choose $\xi(j_1), \ldots, \xi(j_{n-1})$ arbitrarily and then we have $|C|$ choices for $\xi(j_n)$. So
\[
\left< 1_{C}(\xi,\sigma) \right>_{(\xi,\sigma)\in G\wr S_n} = \frac{(n-1)!|G|^{n-1}|C|}{n!|G|^n} = \frac{1}{n}\frac{|C|}{|G|},
\]
as needed. \qed

\subsection{Norms in short intervals}\label{sec:normsinshort}
Here we discuss two $G$-factorization arithmetic functions related to norms, and our results on their mean value in short intervals. For a function field $E/\FF_q(T)$, we define the following arithmetic functions. For $f \in M_{n,q}$, we define \begin{align*}
b_{E/\FF_q(T)}(f) &=
\begin{cases}
1,& \mbox{if }\exists I \subseteq \mathcal{O}_E: (f)  =\mathrm{Norm}_{E/\FF_q(T)} (I),\\
0,& \mbox{otherwise,} \end{cases}\\
r_{E/\FF_q(T)}(f) &= \# \{ I \text{ ideal in }\mathcal{O}_E : \mathrm{Norm}_{E/\FF_q(T)}(I)=(f)\}.
\end{align*}

The number field versions of $r,b$ were studied extensively: Let $E/\QQ$ be a finite extension. 
Odoni \cite[Thm.~1]{odoni1975} computed the asymptotic of the mean value of $b_{E/\QQ}$. When $E/\QQ$ is Galois, the work of Ramachandra \cite{ramachandra1976some} gives the mean value of $b_{E/\QQ}$ in $[x,x+x^{\varepsilon}]$ for some $0<\varepsilon<1$.

Weber \cite{weber1896} computed the mean value of $r_{E/\QQ}$, and studied the error term. We refer to  Bourgain and Watt \cite[Thm.~2]{bourgain2017} for the state-of-the-art result on the error term when $E=\QQ(i)$, and to Lao \cite{lao2010} for more general $E$. 
These results in particular gives the expected asymptotics for the mean value of $r_{E/\QQ}$ in $[x,x+x^{\varepsilon}]$ for some $0<\varepsilon<1$.

In Appendix~\ref{app:1}, we prove a function field analogue of Odoni's result on the average of $b_{E/\FF_q(T)}$ in long intervals when $E/\FF_q(T)$ is Galois. This is to be done in the most general limit $q^n\to \infty$. Appendix~\ref{app:1} also treats $r_{E/\FF_q(T)}$ for which the rationality of the corresponding Dedekind zeta function gives a closed formula for the mean value.

The result to be presented is a computation of the mean values  in short intervals. 

\begin{theorem}\label{thm:norms}
For every $B>0$ there exists a constant $M_B$ satisfying the following property. Let $q$ be a prime power. Let $n>m\geq 2$ if $q$ is odd and $n>m \geq 3$ otherwise. Let $G$ be a finite group and let $E/\FF_q(T)$ be a geometric $G$-extension which has is tamely  ramified at the infinite prime. Let $f_0\in \FF_q[T]$ monic of degree $n$. Assume that $\genus(E), n,|G| \leq B$. Then 
\[
\left<b_{E/\FF_q(T)}(f)\right>_{\deg(f-f_0)\leq m} = 1+O_B( q^{-1/2}),
\]
\[
\left<r_{E/\FF_q(T)}(f)\right>_{\deg(f-f_0)\leq m} = \binom{n+\frac{1}{|G|}-1}{n} + O_{B}( q^{-1/2}).
\]
\end{theorem}

To see how Theorem~\ref{thm:norms} is deduced from Theorem~\ref{thm:maintechnical} we need to express $r,b$ as $G$-factorization arithmetic functions (Example~\ref{example:b}) and to compute the mean value on the wreath product, or alternatively apply the results from Appendix~\ref{app:1}.

Let $E/\FF_q(T)$ be a Galois extension. Given a prime polynomial $P \in \FF_q[T]$, we denote by $g(P;E)$, $f(P;E)$ and $e(P;E)$ the number of distinct primes in $E$ lying above $P$, the inertia degree of $P$ in $E$ and the ramification index of $P$, respectively. 
\begin{lemma}\label{lem:bcriteria}
Let $E/\FF_q(T)$ be a geometric $G$-extension. 
\begin{enumerate}
\item The functions $b_{E/\FF_q(T)}$ and $r_{E/\FF_q(T)}$ are multiplicative.
\item Let $f \in M_{n,q}$ with prime factorization $f = \prod_{i=1}^{k} P_i^{a_i}$. Then 
\begin{equation}\label{eq:bE}
b_{E/\FF_q(T)}(f)=\begin{cases}
1, & \mbox{if $f(P_i;E) \mid a_i$  for all $i$},\\
0, &\mbox{otherwise},
\end{cases}
\end{equation}
and if we put $b_i=a_i/f(P_i;E)$ and $g_i = g(P_i;E)$, then we have  
\begin{equation}\label{eq:rE}
r_{E/\FF_q(T)}(f) = b_{E/\FF_q(T)} (f)\cdot \prod_{i=1}^{k}  \binom{b_i+g_i-1}{g_i-1} .
\end{equation}
\end{enumerate}
\end{lemma}
\begin{proof}
Let $P$ be a prime polynomial and let $\mathfrak{P}$ be a prime ideal of $\mathcal{O}_E$ lying above $P$. Then 
\[
\mathrm{Norm}_{E/\FF_q(T)} \mathfrak{P} = (P)^{f(P;E)}.
\]
By multiplicativity of the norm map and by unique factorization in $\mathcal{O}_E$, it follows that the image of $\mathrm{Norm}_{E/\FF_q(T)}$ on the non-zero ideals in $\mathcal{O}_E$ is the semigroup generated by $\{ (P)^{f(P;E)} \}_{P \in \mathcal{P}_q}$, which establishes \eqref{eq:bE}. 
It now immediately follows that $b_{E/\FF_q(T)}$ is multiplicative. 

If $f_1$ and $f_2$ are relatively prime polynomials, then from unique factorization of ideals in $\mathcal{O}_E$, every ideal $I$ of $\mathcal{O}_E$ with ${\rm Norm}_{E/\FF_q(T)} I = (f_1 f_2)$ has a unique factorization $I=I_1I_2$, with ${\rm Norm}_{E/\FF_q(T)} I_j = (f_j)$. Indeed, if $I=\prod \mathfrak{P}_i^{a_i}$ take $I_j$ be the product of $\mathfrak{P}_i^{a_i}$ with $\mathfrak{P}_i\mid f_j$.
Thus, 
\[
r_{E/\FF_q(T)} (f_1f_2) 
	= \sum_{
		\substack{
			{\rm Norm}_{E/\FF_q(T)} I = (f_1f_2)
		} 
	  }
		1 = \sum_{\substack {{\rm Norm}_{E/\FF_q(T)} I_1 = (f_1) \\ {\rm Norm}_{E/\FF_q(T)}I_2= (f_2)}}1 = r_{E/\FF_q(T)}(f_1) r_{E/\FF_q(T)}(f_2).
\]
This implies that $r_{E/\FF_q(T)}$ is multiplicative. In particular, it suffices to prove  \eqref{eq:rE} for $f=P^a$ a prime power.

If $f(P;E)\nmid a$, then $b_{E/\FF_q(T)}(P^a)=0$, hence also $r_{E/\FF_q(T)}(P^a)=0$. Assume now that $f(P;E)\mid a$ and let $b=a/f(P;E)$.
Let $\mathfrak{P}_1,\ldots,\mathfrak{P}_{g}$ be the primes of $\mathcal{O}_E$ lying above $P$. Since ${\rm Norm}_{E/\FF_q(T)} \mathfrak{P}_j = P^{f(P;E)}$, the solutions to $\mathrm{Norm}_{E/\FF_q(T)}I = P^{a}$, are of the form $I=\prod_{j=1}^{g} \mathfrak{P}_j^{c_j}$ with $c_j\geq 0$ and $\sum_{j=1}^{g} c_j = b$. As there are $\binom{b+g-1}{g-1}$ many such sequences of $c_j$, the proof is done. 
\end{proof}
Lemma \ref{lem:bcriteria} allows us to realize $b_{E/\FF_q(T)}$, $r_{E/\FF_q(T)}$ as $G$-factorization arithmetic functions.
\begin{example}\label{example:b}
Let $\omega \in \Omega_G$, and let $\Sigma \in \omega$. So $\Sigma$ is a coset of a subgroup of $G$, say $\Sigma = \sigma I$. Let $e_{\omega} = |I|$, $f_{\omega} = [ \left< \sigma, I\right> : I]$, and $g_w = |G|/e_wf_w$. Now we define the $G$-factorization arithmetic functions
\begin{equation}\label{eq:defbr}
\begin{split}
b(\lambda) &=
\begin{cases}
1, & \mbox{if }\lambda(d,a,\omega) > 0 \Rightarrow f_{\omega} \mid a, \\
0, & \mbox{otherwise.}
\end{cases}\\
r(\lambda) &= b(\lambda) \cdot \prod_{(d,a,w)} \binom{a/f_w + g_w-1}{g_w-1}^{\lambda(d,a,w)}
\end{split}
\end{equation}

Let  $E/\FF_q(T)$ be a geometric $G$-Galois extension. Then
\begin{equation}\label{eq:b-is-factorization-type}
b_{E/\FF_q(T)}(f) = b(\lambda_{f;E/\FF_q(T)}) \quad \mbox{and} \quad r_{E/\FF_q(T)}(f) = r(\lambda_{f;E/\FF_q(T)}).
\end{equation}
Indeed, by \eqref{eq:bE} and \eqref{eq:rE}, it suffices to note that $w = \left(\frac{E/\FF_q(T)}{P}\right)$, then $e_w=e(P;E)$, $f_w=f(P;E)$, and $g_w = g(P;E)$.
\end{example}

\begin{proof}[Proof of Theorem~\ref{thm:norms}]
By Theorem~\ref{thm:maintechnical}, it suffices to compute the average of the $G$-factorization arithmetic functions $b$ and $r$ given in \eqref{eq:b-is-factorization-type} on the group $G\wr S_n$. For brevity we compute them together by computing the average of  $r^s$ for any $s \in \CC$ (and noting that $r=r^1$ and $b=r^0$). Put $N=|G|$ and let $s \in \CC$. We show that 
\[
\left< r^s_{G\wr S_n}(\xi,\sigma) \right>_{(\xi,\sigma)\in G\wr S_n} = \binom{n+N^{s-1}-1}{n}.
\]

Let $\sigma = \sigma_1\cdots \sigma_r$ be the factorization of $\sigma\in S_n$ to disjoint cycles and let $\xi\in G^{n}$. Recall that if we write $\sigma_i= (j_1\ \cdots \ j_\ell)$, then  $C_{(\xi,\sigma),\sigma_i}$ is defined to be the conjugacy class of the element
\[
 \xi(j_{\ell}) \cdots \xi(j_1).
\]

Let $d,a\geq 1$ and $\omega\in \Omega_G$ with $\lambda_{(\xi,\sigma)}(d,a,\omega)>0$. Then $a=1$ and $\omega=C_{(\xi,\sigma),\sigma_i}$ for some $i$. In particular, $e_\omega=1$, and thus $f_{\omega} = 1$ if and only if $C_{(\xi,\sigma),\sigma_i}=1$, where  $e_\omega$ and $f_{\omega}$ are as defined in Example~\ref{example:b}.

By \eqref{eq:defbr}, we have that  $r^s_{G\wr S_n}(\xi,\sigma)\neq 0$ if and only if $f_\omega\mid a$ for all $(d,a,\omega)$ with $\lambda_{(\xi,\sigma)}(d,a,\omega)>0$. So if $r^{s}_{G\wr S_n}(\xi, \sigma)\neq 0$, then $a=1$, hence $f_{\omega}=1$, and so $C_{(\xi,\sigma),\sigma_i} =1$, for all $i$. As $g_\omega = \frac{N}{e_\omega f_\omega}=N$, and so $\binom{a/e+g-1}{g-1}=N$, we deduce from \eqref{eq:defbr} that 
\begin{equation*}
r^s_{G\wr S_n}(\xi,\sigma) = \begin{cases} \prod_{(d,a,\omega):\lambda_{(\xi,\sigma)}(d,a,\omega)>0} N^{s\lambda_{(\xi,\sigma)}(d,a,\omega)}, & \mbox{if $C_{(\xi,\sigma),\sigma_i}=1$ for all $i$,} \\ 0, & \mbox{otherwise.} \end{cases}
\end{equation*}
Put
\begin{equation*}
X_n = \{ (\xi,\sigma) \in G \wr S_n: C_{(\xi,\sigma),\sigma_i} = 1, \, \forall i\},
\end{equation*}
so that 
\begin{equation}\label{eq:meanwr1}
\left<
r^s_{G\wr S_n}(\xi,\sigma)
\right>_{(\xi,\sigma)\in G\wr S_n} = \frac{\sum_{(\xi,\sigma) \in X_n} (N^s)^{r(\sigma)}}{\#G\wr S_n},
\end{equation}
where $r(\sigma)$ is the number of cycles in $\sigma$. For a fixed $\sigma\in S_n$ with a factorization $\sigma=\sigma_1 \ldots \sigma_{r}$ to disjoint cycles, we have
\begin{equation}\label{eq:meanwr2}
\sum_{\xi \in G^{n}:  (\xi,\sigma) \in X_n} (N^s)^{r(\sigma)} = (N^{s})^{r} N^{n-r} = N^n \cdot (N^{s-1})^r,
\end{equation}
since if $\sigma_i=(j_1 \ \ldots\ j_d)$, then $\xi(j_1), \ldots, \xi(j_{d-1})$ can be chosen arbitrarily and $\xi(j_d)$ must be equal to $\prod_{k=1}^{d-1}\xi(j_k)^{-1}$, so we lose one power of $N$ for each orbit. Plugging  \eqref{eq:meanwr2} in \eqref{eq:meanwr1}, we find that 
\begin{equation}\label{eq:meanwr3}
\left<
r^s_{G\wr S_n}(\xi,\sigma)
\right>_{(\xi,\sigma)\in G\wr S_n} = \frac{\sum_{\sigma \in S_n}  (N^{s-1})^{r(\sigma)}}{\#S_n}.
\end{equation}

We apply the exponential formula for permutations \cite[Cor.~5.1.9]{stanley1999} with $f(i)= N^{s-1}$ the constant function and $h$ defined by $h(0)=1$ and 
\[
h(i) = \sum_{\sigma\in S_i} (N^{s-1})^{r(\sigma)}.
\]
Then the formula gives that 
\[
E(x):= \sum_{i=0}^{\infty} h(i)\frac{x^i}{i!} = \exp(\sum_{i\geq 1} N^{s-1}\frac{x^i}{i}). 
\]
As $\sum_{i \ge 1} x^i/i=-\ln(1-x)$, we can simplify the right hand side using the binomial series to get that
\[
E(x) = (1-x)^{-N^{s-1}} = \sum_{i \ge 0} (-1)^i  \binom{-N^{s-1}}{i}x^i.
\]
In particular, by \eqref{eq:meanwr3} we have 
\begin{equation*}
\left<
r^s_{G\wr S_n}(\xi,\sigma)
\right>_{(\xi,\sigma)\in G\wr S_n} = \frac{h(n)}{n!} =   (-1)^n  \binom{-N^{s-1}}{n}  = \binom{n+N^{s-1}-1}{n},
\end{equation*}
as needed.
\end{proof}

\section{Galois Theory}\label{GT}
\subsection{\texorpdfstring{$G$}{G}-factorization arithmetic functions and the Frobenius automorphism}
\label{sec:Frobenius}
Let $\psi$ be a $G$-factorization arithmetic function and $E/\FF_q(T)$ a $G$-Galois geometric extension. The goal of this section is, for a given $a=(a_0,\ldots, a_{n-1})\in \FF_q^{n}$, to naturally construct an element $\phi_a\in G\wr S_n$ such that 
\begin{equation}\label{eq:Garith-fun-Frob}
\psi_{E/\FF_q(T)}(T^n+a_{n-1}T^{n-1}+\cdots + a_0) = \psi_{G\wr S_n}(\phi_a).
\end{equation}
We start with a general construction which  we later specialize to our setting.
Let $F$ be a field and $\pi\colon C\to \AAA^1_{F}$ a branched covering of smooth geometrically connected $F$-curves with function field extension $E/F(T)$.
Assume that $E/F(T)$ is Galois with Galois group $G$. 

This gives rise to the following cover of varieties with corresponding function fields
\begin{equation}\label{eq:GeoCon}
\xymatrix{
C^n\ar[d]^{\pi^n} && E_1\cdots E_n\ar@{-}[d]\\
\AAA^n_{F}\ar[d]^{s}&&F(Y_1,\ldots, Y_n)\ar@{-}[d]\\
\AAA^{n}_{F} = \AAA^n/S_n && F(A_0,\ldots, A_n).
}
\end{equation}
Here $S_n$ acts on $\AAA^n$ by permuting the coordinates: if $(Y_1,\ldots, Y_n)$ are the coordinates of $\AAA^n$ and $(A_0,\ldots, A_{n-1})$ of $\AAA^n=\AAA^n/S_n$, then the  map $s$ is given by
\[
A_0 = (-1)^nY_1\cdots Y_n\quad,\quad  \ldots\quad,\quad A_{n-1} = -(Y_1+\ldots+Y_n).
\]
Also, $E_i$ is the function field of the $i$-th copy of $C$ in $C^n$, in particular, for every $i$ there exists an isomorphism
\begin{equation}\label{eq:isomorphisms}
\varphi_i \colon E_i\to E
\end{equation}
with $\varphi_i(Y_i)=T$ that fixes $F$. Put $\varphi_{i,j}\colon E_i\to E_j$ to be $\varphi_j^{-1}\circ \varphi_{i}$.
Let $D_1(T) \FF_q[T]$ be the discriminant ideal of $\pi$ and $D_2(A_0,\ldots, A_{n-1}) = \disc_T(T^n + A_{n-1}T^{n-1} + \ldots + A_0)$ and put
\begin{equation}\label{def:D}
D(A_0,\ldots, A_{n-1} ) = D_2(A_0,\ldots, A_{n-1})\prod_{i} D_1(Y_i) \in F[A_0,\ldots, A_{n-1}]
\end{equation}
which is a non-zero polynomial in the $A_i$-s. Then, for a point $a\in \AAA^{n}(F)$ we have
\begin{equation}\label{property1:D}
D(a)\neq 0 \Longrightarrow \mbox{ $a$ is unramified in $C^n$}.
\end{equation}
If we write $f(T) = T^n+a_{n-1}T^{n-1}+\cdots +a_0$, then the condition $D(a)\neq 0$ is equivalent to $f$ being a separable polynomial that does not vanish on the branch points of $\pi$ which are exactly the roots of $D_1$. The Riemann-Hurwitz formula (see e.g.\ \cite[Thm.~3.6.1]{fried2005}) gives that $\deg D_1 \ll \genus(C)+|G|$. On the other hand, $\deg D_2\ll n$. So, if  $B\geq \max\{\genus(C),|G|,n\}$ then $\deg D$ is bounded in terms of $B$.

The extension $E_1\cdots E_n/F(A_1,\ldots, A_n)$ is a Galois extension with Galois group isomorphic to $G\wr S_n$.
More explicitly, the action of an element  $(\xi,\sigma) \in G\wr S_n$ on $E_1\cdots  E_n$ is given by
\begin{equation}\label{eq:action}
\begin{split}
(\xi,\sigma). e_i &= \xi(\sigma(i))(\varphi_{i,\sigma(i)}(e_i)), \qquad e_i\in E_i.
\end{split}
\end{equation}
This is compatible with the imprimitive action \eqref{eq:imprimitiveaction}.

If $F=\FF_q$ is a finite field, then any point $a\in \AAA^n(F)$  with $D(a)\neq 0$ induces a Frobenius conjugacy class $\phi_a \subseteq G\wr S_n$, which is the higher dimensional version of \eqref{eq:FrobatP} and is defined similarly. Now we can prove \eqref{eq:Garith-fun-Frob}:

\begin{proposition}\label{prop:frob}
Let $F = \FF_q$, let $f(X) = T^n+a_{n-1}T^{n-1} + \cdots + a_0\in \FF_q[T]$ such that the point $a=(a_0,\ldots, a_{n-1})$ is unramified in $C^n$ and let $\phi_a \subseteq G\wr S_n$ be the Frobenius conjugacy class. Then \eqref{eq:Garith-fun-Frob} holds for every $\psi\in \Lambda^*$ .
\end{proposition}

\begin{proof}
To ease notation we identify each of the $E_i$ with $E$ via the map $\phi_i$.
Let $f = P_1 \cdots P_r$ be the prime factorization of $f$ with $P_i$ monic irreducible of degree $d_i$. For each $i=1,\ldots,r$, let $\alpha_{i,1}\in \AAA^1(\overline{\FF}_q)$ be a root of $P_i$  and  $\beta_{i,1}\in C(\overline{\FF}_q)$ with $\alpha_{i,1}=\pi(\beta_{i,1})$ and let $\alpha_{i,j} = \alpha_{i,1}^{q^{j-1}}$ be the other roots, and respectively $\beta_{i,j} = \beta_{i,j}^{q^{j-1}}$, $j=1,\ldots, d_i-1$. We replace the indices of $C^n$ and of the middle $\AAA^n$ in \eqref{eq:GeoCon} to be $I = \{(i,j) : i=1,\ldots, r, \ j=1,\ldots, d_{i}\}$.

So $(\beta_{i,j})_{(i,j)\in I} \in C^{n}(\overline{\FF}_q)$ maps under $\pi^n$ to $(\alpha_{i,j})_{(i,j)\in I}\in \AAA^n(\overline{\FF}_q)$ which maps under $s$ to $a=(a_0,\ldots, a_{n-1})\in \AAA^{n}(\FF_q)$.

To this end, let $\phi_a$ be the corresponding Frobenius element of $(\beta_{i,j})_{(i,j)\in I}$ and let $h\in E_{i,1}$ be a rational function that is regular at all $\beta_{i,j}$.
Then, by definition,
\[
(\phi_a^{d_i} h)(\beta_{i,1}) = (h(\beta_{i,1}))^{q^{d_i}}.
\]
Write $\phi_a = (\xi,\sigma)$; then $\phi_a^{d_i} = (\prod_{k=1}^{d_i} \xi^{\sigma^{k-1}},\sigma^{d_i})$.
The coordinate $\sigma  \in S_n$ is induced from the action of the Frobenius automorphism on the roots of $f$, so since $\alpha_{i,j}=\alpha_{i,1}^{q^{j-1}}$, we have that $Y_{i,j} = \sigma^{j}(Y_{i,1})$, $j=0,\ldots, d_{i}-1$ and $Y_{i,1}= \sigma^{d_i} (Y_{i,1})$. The latter also implies that $E_{i,1}$ maps to itself under $\sigma^{d_i}$, hence
\[
(\phi_a^{d_i} h) (\beta_{i,1}) =( \prod_{k=1}^{d_i} \xi^{\sigma^{k-1}}(i,1) h)(\beta_{i,1}) = (\prod_{j=1}^{d_i} \xi(i,j) h) (\beta_{i,1}).
\]
To conclude, we obtained
\[
(\prod_{j=1}^{d_i} \xi(i,j) h) (\beta_{i,1}) =  (h(\beta_{i,1}))^{q^{d_i}},
\]
but the Frobenius at $P_i$ in $E$ is the unique element of $G$ satisfying this, so ${\rm Frob}_{P_i} = \prod_{j=1}^{d_i} \xi(i,j)$. Thus, $\lambda_{\phi_a}=\lambda_{f;E/\FF_q(T)}$, which completes the proof.
\end{proof}

\subsection{Computation of a Galois group }
We keep the notation as in \S\ref{sec:Frobenius}, in particular $F$ is a field and $\pi\colon C\to \AAA^1_F$ is a branched covering of geometrically irreducible $F$-curves that is generically Galois with Galois group $G$.
For $a=(a_0,\ldots, a_{n-1})\in \AAA^{n}$ and for $0\leq m<n$, we consider the following subspace of $\AAA^{n}$ 
\begin{equation}\label{eq:Wam}
W = W_{a,m} = \{ (w_0,\ldots, w_{n-1}) \in \AAA^n : w_i=a_i,\ i=m+1,\ldots,  n-1\}\cong \AAA^{m+1}.
\end{equation}
So if $\AAA^n$ is the space of coefficients of polynomials, then $W$ is the short interval $I(T^n+a_{n-1} T^{n-1} + \cdots + a_0,m)$. 
Let $U$ and $V$ be irreducible components of $(s\circ \pi^n)^{-1}(W)$ and $s^{-1}(W)$, respectively. Let $M$, $L$, and $K$ be the function fields of $U$, $V$, and $W$,
\begin{figure}[h]\label{figure1}
\[
\xymatrix{
C^n\ar[d]^{\pi^n} &U\ar[l]\ar[d]^{\pi^n}
				 && M\ar@{-}[d]\\
\AAA^n_{F}\ar[d]^{s}&V\ar[l]\ar[d]^{s}
				&&L = K(y_1,\ldots, y_n)\ar@{-}[d]\\
\AAA^{n}_{F}  & W\ar[l]
				&& K = F(A_0,\ldots, A_m)
}
\]
\caption{Variety and field diagrams}
\label{construction_of_fiberproducts}
\end{figure}
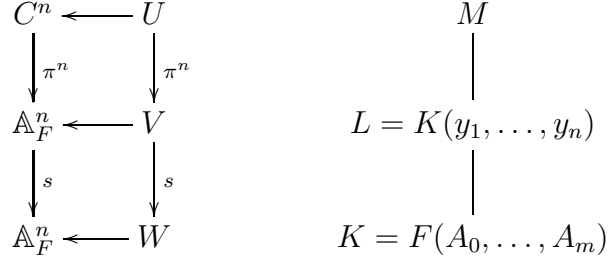
where $A_0, \ldots, A_m$ are independent variables and the $y_i$-s satisfy
\[
f(T) = T^n + a_{n-1}T^{n-1} + \cdots + a_{m+1}T^{m+1} + A_mT^m + \cdots +A_0 = \prod_{i=1}^n (T-y_i).
\]
Then, $M/K$ is a Galois extension and if $D(a_{n-1},\ldots, a_{m+1},A_{m},\ldots, A_0)\neq 0$, then by \eqref{property1:D}, it is unramified, hence  its Galois group  is canonically isomorphic to the subgroup the generic Galois group $G\wr S_n$ given in \eqref{eq:action}; namely  all elements that generically map  $U$ to itself. We identify $\Gal(M/K)$ with this subgroup, in particular 
\[
H=\Gal(M/L) \leq G^n.
\] 

We have that $\Gal(M/K)$ equals $G\wr S_n$ if and only if $(s\circ \pi^n)^{-1}(W)$ is irreducible, i.e. equals to $U$.

We prove that this is indeed the case if the ramification at infinity is tame.

\begin{proposition}\label{general_computation}
Under the notation above, assume that $m\geq 2$ if $q$ is odd and $m \geq 3$ if $q$ is even. Further assume that the infinite prime is tamely ramified in the fixed field $E^{ab}$ in $E$ of the commutator of $G $. Then
\[
\Gal(M/K) = G\wr S_n.
\]
\end{proposition}

\begin{remark}
	Proposition~4.6 in \cite{bank2018} coincides with the special case of Proposition~\ref{general_computation} where $G=\ZZ/2\ZZ$, $C=\AAA^1$ and $\pi(x)=x^2$. Cyclic extensions with genus $0$ were partly treated by Cohen \cite[Thm.~12]{cohen1980}.
\end{remark}

\subsubsection{Reduction steps}
Since by \cite[Proposition 3.6]{bank2015a}, $\Gal(L/K) = S_n$ and since $\Gal(M/K)\leq G\wr S_n$, in order to prove that $\Gal(M/K) = G\wr S_n$ it suffices to show that  $H=G^n$. 

The proof of Proposition~\ref{general_computation} is reduced, by elementary finite group theory, to the following statements:
\begin{lemma}\label{computation_abelian}
Proposition~\ref{general_computation} holds true if $G$ is abelian.
\end{lemma}

\begin{lemma}\label{computation_two-fold}
For each $i\neq j$ the projection on the $i,j$-th coordinates $H\to G^n\to G^2$ is surjective.
\end{lemma}

\begin{proof}[Proof that  Lemmas~\ref{computation_abelian} and \ref{computation_two-fold} imply Proposition~\ref{general_computation}]
From Lemma~\ref{computation_abelian} applied to the abelianization $G^{{\rm ab}}$ of $G$, we get that $H$ surjects onto $(G^{{\rm ab}})^n$. This in particular means that $H$ is contained in no proper normal subgroup with abelian quotient. By Lemma~\ref{computation_two-fold}, $H$ surjects onto any projection to two coordinates. To finish the proof we need that these two group theoretical properties suffice to imply that $H=G^n$ and this is indeed the case, as the lemma below shows. 
\end{proof}

\begin{lemma}
Let $H$ be a subgroup of $G^n$. Suppose that $H$ maps surjectively onto $G^2$ under each possible projection onto two copies. If $H$ is a proper subgroup of $G^n$, then it is contained in a proper normal subgroup with abelian quotient. 
\end{lemma}

\begin{proof}
The case $n=2$ is trivial, and by induction, we may assume this is true for $n-1$ and that $n\geq 3$. If any map from $H$ to the product of $n-1$ copies of $G$ is not surjective, then the image of $H$ is contained in a normal subgroup with abelian quotient, so $H$ is as well. So we may assume that the projections from $H$ to any product of $n-1$ copies of $G$ are surjective.  Now apply Goursat's lemma to $G^{n-1}$ and $G$. We get that there is a group $G'$, surjections $a\colon G^{n-1} \to G'$ and $b\colon G\to G'$, such that $H$ consists of tuples $(g_1,\ldots,g_n)$ in $G^n$ with $a(g_1,...,g_{n-1})=b(g_n)$. Moreover since $H$ is a proper subgroup, $G'$ is non-trivial. 

Now because the map $H\to G^{n-1}$ obtained by dropping the $i$-th coordinate is surjective, for any $g_n$ there exists some $g_i$ such that $(e,\ldots, e, g_i,e,\ldots, g_n)\in H$, and so $a(e,\ldots,e,g_i,\ldots,e)=b(g_n)$.  Putting $G_1 = \{(g_1,e,\ldots,e) : g_1\in G\}$ and $G_2= \{ (e,g_2,e,\ldots ,e ): g_2\in G\}$, we conclude that $a$ maps both $G_1$ and $G_2$  onto $G'$. Thus as $G_1$ and $G_2$ commute, we get that $G'$ must be abelian. 
Thus the pre-image of $G'$ is the desired normal subgroup with abelian quotient.
\end{proof}

To finish the proof of Proposition~\ref{general_computation}, it remains to prove Lemmas~\ref{computation_abelian} and \ref{computation_two-fold}. Since the former lemma is technical, we start by proving the latter assuming the former.

\subsubsection{Proof of Lemma~\ref{computation_two-fold} using Lemma~\ref{computation_abelian}}

We look at the covering $\Upsilon$ of $W$ defined by adjoining two roots $y_i,y_j$ of the polynomial; i.e., $\Upsilon$ is the quotient space of $V$ under the action of $S_{n-2}$, so $\Gal(V/\Upsilon)\cong S_{n-2}$, the group of all permutations fixing $i,j$. Let $\Gamma$ be the Galois group of $U/\Upsilon$. So, the restriction-of-automorphisms map induces a surjection $\Gamma\to \Gal(V/\Upsilon)=S_{n-2}$. 

	The covering $\Upsilon$ maps to $\AAA^2$ by sending to the two roots. The fibers are connected because we just add two congruence conditions. We have the covering $C^2\to \AAA^2$ with Galois group $G^2$, and its fiber product $\Upsilon\times_{\AAA^2}(C^2)$ is geometrically connected, since the fiber of $\Upsilon \to \AAA^2$ is geometrically connected. This implies that $\Gamma$ surjects onto $G^2$.

We apply Goursat's lemma to these two maps. They are jointly surjective unless some quotient of $G^2$ matches some quotient of $S_{n-2}$. But all normal subgroups of $S_{n-2}$ are contained in $A_{n-2}$, so this can only happen if there is some non-trivial relation with order two quotients of $G$.
This is not possible by the abelian case, hence the proof is done. 
\qed

The proof of Lemma~\ref{computation_abelian} is more technical and requires some preparation.

\subsubsection{Some more group theory}
This section contains well known facts that we summarize for the convenience of the reader. 
We start by stating two well known facts on the symmetric group. The first is on normal subgroups:
\begin{lemma}\label{lem_Sn}
	Let $n \ge 1$. The group $S_n$ does not have any normal subgroups of odd prime index.
\end{lemma}
The second is about the invariant subspaces of the standard representation of $S_n$ on $\FF_p^n$ acting by permuting the coordinates.

\begin{lemma}\label{lem_invsub}  The invariant subspaces of $\FF_p^n$ under  $S_n$ ($n\geq 3$) are:
	\begin{enumerate}
		\item $V_0=\{(0,\ldots,0)\}$,
		\item $V_1 = \mathrm{sp}_{\FF_p} \{(1,\ldots,1)\}$,
		\item $V_{n-1} = \{(x_1,\ldots,x_n) \in \FF_p^n  :  \sum_{i=1}^n x_i = 0 \}$, and
		\item $V_n = \FF_p^n$.
	\end{enumerate}
\end{lemma}

Recall that the Frattini subgroup $\Phi(G)$  of a group $G$ is defined by $\Phi(G) = \bigcap_{U\leq_{m} G} U$, where the intersection is over the maximal subgroups of $G$. It has the property that for every $H\leq G$, if $H/H\cap \Phi(G) = G/\Phi(G)$, then $H=G$.
If $G$ is finite and $p\mid |G|$, then the subgroup $\Phi_p(G) = [G,G]G^p$ generated by commutators and $p$-th power of elements is normal and $G/\Phi_p(G) \cong (\ZZ/p\ZZ)^r$. Thus,
\begin{equation}\label{eq:Frattini_factorization}
\Phi_p(G) = U_1\cap \cdots \cap U_{r},
\end{equation}
with $U_i$ the kernel of the projection on the $i$-th coordinate, so $U_i$ is normal in $G$ of index $p$.

\begin{lemma}\label{Kummer)general_k_lemma}
	Let $G$ be a finite abelian 
	group and $H\leq G$. Assume that $H/(H\cap \Phi_p(G)) \cong G/ \Phi_p(G)$ for every $p \mid |G|$. Then $H=G$.
\end{lemma}
\begin{proof}
Since $G$ is abelian, $\Phi(G) =\bigcap_{p \mid |G|} \Phi_p(G)$ and $G/\Phi(G) \cong \prod_{p \mid |G|} G/\Phi_p(G)$. So the assumption gives that $H/(H\cap \Phi(G)) = G/\Phi(G)$ and so $H=G$.
\end{proof}

Let $L$ be a field and $p$ a prime. If $p \nmid \mathrm{char}(L)$ let $\wp(x)=x^p$ and $L^{\circ}=L^{\times}$, otherwise let $\wp(x)=x^p-x$ and $L^{\circ}=L$. We say that elements in  $L^{\circ}$ are $p$-independent if they are linearly independent in $L^{\circ}/\wp(L^{\circ})$, considered as a $\FF_p$-vector space.
\begin{lemma}\label{lem:fieldfrattini}
Let $G$ be a finite abelian group and $H \leq G$. Let $L$ be a field such that for every $p \mid |G|$, either $L$ contains a primitive $p$-th root of unity or $L$ is of characteristic $p$. Let $M/L$ be an $H$-Galois extension. 
For a prime divisor  $p$ of $ |G|$, put $U_1,\ldots, U_{r(p)}$ as in \eqref{eq:Frattini_factorization}. 
Then, for every $1\leq i\leq r$ there exists $\alpha_{i,p}\in L$ such that 
\[
M^{H\cap U_i} = L(\beta_{i,p}), \qquad \wp(\beta_{i,p})=\alpha_{i,p}. 
\] 
Moreover, if  $\alpha_{1,p}, \ldots, \alpha_{r(p),p}$ are $p$-independent for all $p\mid |G|$, then $H= G$.
\end{lemma}

\begin{proof}
First we note that $H/(H\cap U_i)\leq G/U_i = \ZZ/p\ZZ$ so $\Gal(M^{H\cap U)i}/L)\leq \ZZ/p\ZZ$, and Kummer theory (if $\mathrm{char}(L) \neq p$) or Artin-Schreier theory (otherwise) give us the required elements $\alpha_{i,p}$. Now, if the $\alpha_{i,p}$-s are $p$-independent, then by \eqref{eq:Frattini_factorization}, Kummer theory and Artin-Schreier theory,  we find that 
\[
(\ZZ/p\ZZ)^{r(p)} \cong \Gal(M^{H \cap \Phi_p(G)}/L) \cong H/(H\cap \Phi_p(G)).
\]
So by Lemma~\ref{Kummer)general_k_lemma}, $H=G$.
\end{proof}

\subsubsection{Rational functions}
We borrow the following  from \cite[Lem.~4.5]{bank2018}.
\begin{lemma}\label{LemBBFDisc}
Let $\tilde{f}(T) \in K[T]$ be a separable polynomial and let $f(T) = \tilde{f}(T)+A\in K(A)[T]$ where $A$ is transcendental over $K(T)$. Then $\disc(f) \in K[A]$ is not divisible by $A$.
\end{lemma}

\begin{lemma}\label{LemSep}
Let $F$ be a field and let $\mathbf{A}=(A_1,\ldots,A_m)$ be an $m$-tuple of variables ($m \ge 2$). Let $\mathbf{\alpha}=(\alpha_1,\ldots,\alpha_m)$ be an $m$-tuple of scalars from $F$. Let $f_0(T)\in F[T]$ be a polynomial of degree $>m$. Then $\mathcal{F}(\mathbf{A},T)=f_0(T) + \sum_{i=1}^{m} A_i T^i + \sum_{i=1}^{m} A_i \alpha_i$ is separable in $T$.
\end{lemma}

\begin{proof}
It suffices to show that $\mathcal{F}$ is irreducible in $T$, because $\mathcal{F}'$ is linear in $A_1$ and in particular non-zero.
Since $\mathcal{F}$ is primitive in $T$, it is irreducible in $T$ if and only if it is irreducible in $R=F(A_2,\ldots, A_m)[A_1,T]$ by Gauss's lemma.
Since
\[
\mathcal{F} = (T+\alpha_1) A_1 + \mathcal{G},
\]
with $\mathcal{G}=f_0(T) + \sum_{i>1} A_i(T^{i}+\alpha_i)$, either  $\mathcal{F}$ is  primitive in $A_1$, then again by Gauss it is irreducible in $A_1$ and thus in $T$, or $T+\alpha_1$ divides $\mathcal{G}$ in $R$.
In the latter case, $\mathcal{H}= A_1+\frac{\mathcal{G}}{T+\alpha_1}$ is irreducible in $T$ (again primitivity and linearity) and $\deg_{A_2}\frac{\partial\mathcal{H}}{\partial T}=1$ hence it is  non-zero, so $\mathcal{H}$ is separable in $T$.  Moreover, $\mathcal{H}|_{T=-\alpha_1} \neq 0$, so we get that $\mathcal{F} = (T+\alpha_1) \mathcal{H}$ is separable, as needed.
\end{proof}

\begin{lemma}\label{LemSym}

Let $F$ be an algebraically  closed field of characteristic $p$. Let $\mathbf{A}=(A_0,\ldots,A_m)$ be an $(m+1)$-tuple of variables, $m \ge 1$. Let $f_0(T) \in F[T]$ be a monic polynomial of degree $n>m$. Let $f(T) = f_0(T) + \sum_{i=0}^{m} A_iT^i$ be a polynomial with coefficients in $K=F(\mathbf{A})$. Let $L$ be the splitting field of $f(T)=\prod_{i=1}^{n} (T-y_i)$.

Let $D(T)=\frac{r_1(T)}{r_2(T)} \in F(T)^{\times}$ be a reduced rational function with $\deg r_2 \ge \deg r_1$, and $r_2(T) = c\prod_{j=1}^{d} (T-\alpha_j)$ ($c \in F^{\times}, \alpha_j \in F$). We have
\begin{equation}\label{eq:genArtin}
\sum_{i=1}^{n} D(y_i) = \frac{h(\mathbf{A})}{\prod_{j=1}^{d} f(\alpha_j)}
\end{equation}
where $h \in F[\mathbf{A}]$ is coprime to $\prod_{j=1}^{d} f(\alpha_j)$ as a polynomial in $A_0$.
\end{lemma}
\begin{proof}
We first prove \eqref{eq:genArtin} in the special case $D(T) = 1/(T-\alpha)^k$. Let 
\begin{equation}\label{eq:GExplicit}
g(T) := \frac{f(\frac{1}{T}+\alpha)T^n}{f(\alpha)}=\frac{f_0(\frac{1}{T}+\alpha)T^n + \sum_{i=0}^{m} A_i (1+\alpha T)^iT^{n-i}}{f(\alpha)}.
\end{equation}
We have
\begin{equation}\label{eq:g}
g(T) = \prod_{i=1}^{n} (T-\frac{1}{y_i-\alpha}).
\end{equation}
By Newton's identities, if $p_k := \sum_{i=1}^{n} \frac{1}{(y_i-\alpha)^k}$ and $e_i:=\sum_{1\le a_1 < a_2 < \ldots < a_i \le n} \prod_{j=1}^{i} \frac{1}{y_{a_j}-\alpha}$, then
\begin{equation}\label{eq:NewtonGirard}
p_k =\sum_{\substack{\nu_1+2\nu_2 +\ldots +k\nu_k=k\\ \nu_1\geq0,\ldots, \nu_k\geq 0}}(-1)^k \frac{k(\nu_1+\ldots+\nu_k-1)!}{\nu_1!\nu_2!\ldots \nu_k!} \prod_{i=1}^{k}(-e_i)^{\nu_i},
\end{equation}
and the coefficients are in fact integers. By \eqref{eq:GExplicit} and \eqref{eq:g}, $e_j$ has denominator $f(\alpha)$ and numerator independent of $A_0,\ldots,A_{j-1}$.
Thus all the summands in \eqref{eq:NewtonGirard}, except $e_1^k$, are of the form $\frac{s(\mathbf{A})}{f(\alpha)^j}$ for some $j<k$ and $s\in F[A_1,\ldots,A_m]$. Moreover, $e_1^k$ has denominator $f(\alpha)^k$ and non-zero numerator (it is $(-(f_0'(\alpha)+\sum_{i=1}^{m} iA_i \alpha^{i-1}))^k$, which depends on $A_1$). From \eqref{eq:NewtonGirard} we establish \eqref{eq:genArtin} with $D=\frac{1}{(T-\alpha)^k}$. 

To prove \eqref{eq:genArtin} for general $D$, we write $r_2$ as $c\prod_{i=1}^{e} (T-\beta_i)^{k_i}$, where $\beta_i \in F$ are distinct. The partial fraction decomposition of $D$ is given by
\begin{equation}\label{eq:DPartial}
D(T) = c_0 + \sum_{i=1}^{e} \sum_{j=1}^{k_{i}} \frac{c_{i,j}}{(T-\beta_i)^j}, \qquad (c_0, c_{i,j} \in F, c_{i,k_i} \neq 0).
\end{equation}
Applying \eqref{eq:genArtin} to each summand in \eqref{eq:DPartial} and summing, we obtain \eqref{eq:genArtin} in its generality.
\end{proof}

\begin{lemma}\label{Lem:CopDenom}
Let $F$ be a field of characteristic $p$, $K=F(A_0)$ a field of rational functions and $\wp(x)=x^p-x$. Suppose that 
\begin{equation}\label{eq:FracEquiv}
\frac{a}{b} \equiv \frac{c}{d} \bmod \wp(K)
\end{equation}
where both fractions are in reduced form, and that $b,d$ coprime. Then $b,d$ are $p$-th powers. 
\end{lemma}
\begin{proof}
From \eqref{eq:FracEquiv}, $\frac{ad-cb}{bd} = z^p-z$ for $z \in K$. Writing $z=\frac{z_1}{z_2}$ in reduced form, it follows that the denominator of $z^p-z$ is a perfect $p$-th power, and so $bd$ is a perfect $p$-th power, and the conclusion follows since $b,d$ are coprime.
\end{proof}

\subsubsection{Proof of Lemma~\ref{computation_abelian}}
\label{sec:proofgeneralkummer}
Let $\bar{F}$ be an algebraic closure of $F$. Since
\[
\Gal(M\bar{F} / K\bar{F}) \leq \Gal(M/K) \leq G\wr S_n,
\]
it suffices to prove that $\Gal(M\bar{F} / K\bar{F}) \cong G\wr S_n$. Therefore we may assume w.l.o.g.\ that $F=\bar{F}$. In particular, if $p \mid |G|$ and $p \neq \mathrm{char}(F)$, the field $F$ contains a primitive $p$-th root of unity.

As explained before Lemma~\ref{computation_abelian} it suffices to prove that $\Gal(M/L) = G^n$. To do this we apply Lemma~\ref{lem:fieldfrattini} to $M/L$ with the group $G^n$ instead of $G$.

For a prime $p \mid |G|$ with $p \nmid \mathrm{char}(F)$, let $D_1(T), \ldots, D_r(T)\in F[T]$ be $p$-powerfree polynomials that are $p$-independent such that
\[
E^{\Phi_p(G)} = L(\sqrt[p]{D_1(T)},\ldots, \sqrt[p]{D_r(T)}).
\]
If $p \mid \mathrm{char}(F), |G|$, let $D_1(T),\ldots,D_r(T) \in F(T)$ be rational functions that are $p$-independent such that
\begin{equation}\label{eq:ArtinSetup}
E^{\Phi_p(G)} = L(\beta_1,\ldots,\beta_r), \qquad \beta_i^p-\beta_i = D_i(T). 
\end{equation}
So the $\alpha_{i,p}$-s of Lemma~\ref{lem:fieldfrattini} can be taken to be $D_i(y_j)$ with $i=1,\ldots, r$ and $j=1,\ldots, n$.
Then, it suffices to prove that the $D_{i}(y_j)$ are $p$-independent to finish the proof. 
We separate this part into two cases depending on whether $\mathrm{char}(F)=p$ or not.

\vspace{5pt}
\noindent \textbf{Case A: $\mathrm{char}(F)\neq p$}

\vspace{5pt}
\noindent \textbf{Step 1:} $r=1$.

Put $D = D_1$ and $w_j= D(y_j)$. Let $V$ be the space of linear dependencies of the $w_j$-s:
\begin{equation}\label{V_def_Kummer}
V= \{(v_1,\ldots, v_n) \in \FF_p^n :  w_{1}^{v_1}\cdots  w_{n}^{v_n} \equiv 1 \mod (L^{\times})^p\}.
\end{equation}
We need to prove that $V=0$. Since $V$ is an invariant subspace of $\FF_p^n$ under the action of $S_n=\Gal(L/K)$,  by Lemma~\ref{lem_invsub}  it suffices to prove that $V\neq V_1,V_{n-1}, V_n$.

\vspace{5pt}
\noindent \textbf{Sub-step 1a:} $(1,\ldots, 1) \not\in V$; hence $V\neq V_1,V_{n}$.

	We assume in contradiction that $(1,\ldots,1) \in V$.  In other words, there exists $z\in L$ such that
	\begin{equation}\label{eq:zp=prodDi}
	D(y_1) \cdots  D(y_n) = z^p.
	\end{equation}
	We factor $D$ over $F$ (recall that we reduced to the case $F=\bar{F}$):
	\begin{equation}\label{eq:D}
	D(T) = c \prod_{j=1}^{d} (T-\alpha_j), \qquad \alpha_j \in F, c \in F^{\times}.
	\end{equation}
	Since
	\begin{equation*}
	(-1)^n f(\alpha_j) = \prod_{i=1}^{n} (y_i - \alpha_j) ,
	\end{equation*}
	and using  \eqref{eq:zp=prodDi} and \eqref{eq:D} we obtain
	\begin{equation}\label{dthroughf}
	z^p = D(y_1)  \cdots  D(y_n) =c^n
			 \prod_{j=1}^{d} \prod_{i=1}^{n} (y_i - \alpha_j) = (-1)^{nd} c^n \prod_{j=1}^{d} f(\alpha_j).
	\end{equation}
	In particular, $K(z)/K$ is a Galois subextension of the $S_n$-extension $L/K$ of degree $p$ or $1$. Put $H=\Gal(L/K(z))$ so that by Lemma~\ref{lem_Sn},  either $H=1$, $H=S_n$ or $H=A_n$, where the latter is possible only if $p=2$.
	
	If $H = 1$, then $K(z) = L$, so $n! = 1$ or $n! = p$, which contradicts  $n > 2$. Thus, $H\neq 1$.
	
	Now we show that $H\neq S_n$. If $H=S_n$, then $[K(z):K]=1$. Therefore $z \in K$, as such $z$ is a rational function in the $A_i$-s. From \eqref{dthroughf}, it follows that $\prod_{j=1}^{d} f(\alpha_j)$ is a $p$-th power in $K$. Each $f(\alpha_j)$, as a polynomial in $A_0$, is linear with leading coefficient $1$, so it must appear a multiple of $p$ times. On the other hand, by comparing the coefficient of $A_1$, the equality
$f(\alpha_j)=f(\alpha_k)$  implies that $\alpha_j=\alpha_k$. As $D$ is $p$-th powerfree in $F[T]$ by assumption, we arrive to contradiction.
	
	So $H = A_n$ and $p=2$ and in particular the characteristic is $\neq 2$.
	 There is a unique field $K'$ such that $K \subseteq K' \subseteq L$ with $\Gal(L/K') = A_n$, namely $K' = K(\sqrt{\disc(f)})$. Thus, $ K(z) = K(\sqrt{\disc(f)})$, and so by \eqref{dthroughf}
	 \begin{equation}\label{eq:falphajdisc}
		(-1)^{nd} c^n  \prod_{j=1}^{d} f(\alpha_j) \cdot \disc(f) \in (K^{\times})^2.
		\end{equation}
	The linear-in-$A_0$ polynomial $f(\alpha_j)$ is coprime to $\disc(f)$. Indeed, put $A = f(\alpha_j)$ and apply Lemma~\ref{LemSep} to $\tilde{f}(T)= f(T)-A$, to obtain that $\tilde{f}(T)$ is separable in $T$ and thus by Lemma~\ref{LemBBFDisc},  $\disc(f)$ is not divisible by $A=f(\alpha_1)$, hence coprime to it, as needed. But then $\disc(f)$ is a square by \eqref{eq:falphajdisc}, which contradicts the fact that $\Gal(L/K)=S_n$.

\vspace{5pt}
\noindent \textbf{Sub-step 1b:} $(1,-1 , 0, \ldots , 0) \notin V$; hence $V \neq V_{n-1}$ and $V \neq V_n$.

Assume in contradiction that $(1,-1,0,...,0) \in V$. So, there exists $z \in L$ such that
\begin{equation}\label{dy_1=dy_2^p_Kummer}
D(y_1)=D(y_2)  z^p.
\end{equation}
Consider the following diagram of fields
\[
\xymatrix{
	L \ar@{--}@/^1pc/@<2pc>[dd]^{S_{n-2}} \ar@{-}[d] \ar@/_1.5pc/@{--}@<-2pc>[dddd]_{S_n} \\ K(y_1 , y_2)(z) \ar@{-}[d]^{1 \text{ or } p} \\
	K(y_1 , y_2) \ar@{-}[d]^{n-1} \\
	K(y_1) \ar@{-}[d] \\
	K}
\]
with $\Gal(L/K(y_1,y_2)) = S_{n-2}$ and  $K(y_1,y_2)(z)/K(y_1,y_2)$ Galois  of degree $1$ or $p$.

Assume in contradiction that  $[K(y_1,y_2)(z):K(y_1,y_2)]=1$, then $z\in K(y_1,y_2)$. Applying the norm map
\begin{equation}\label{nfirstnorm}
N\colon K(y_1,y_2)\to K(y_1)
\end{equation}
on \eqref{dy_1=dy_2^p_Kummer}, multiplying by $D(y_1)$, and considering \eqref{dthroughf}, we obtain
\begin{equation} \label{equation_norms_Kummer} D(y_1)^n=D(y_1) D(y_2) \cdots D(y_n) {N(z)^p}=  (-1)^{dn}c^n \prod_{j=1}^{d} f(\alpha_j) {N(z)^p}.\end{equation}
The field $K(y_1)$ is the field of rational functions in $A_0,A_2,\ldots ,A_m,y_1$ over $F$ since $A_1 = -\frac{A_0 + A_2y_1^2 + \cdots }{y_1}$.

This implies that $f(\alpha_j)$ and $f(\alpha_k)$, as elements in $F(A_2,\cdots,A_m,y_1)[A_0]$ are associate if and only if $\alpha_j=\alpha_k$. 
Since $D$ is $p$-th powerfree, for every $j$ the multiplicity of $f(\alpha_j)$ in the right hand side product in \eqref{equation_norms_Kummer} is $\not\equiv 0 \mod p$. On the other hand, on the left hand side of \eqref{equation_norms_Kummer} the multiplicity of $f(\alpha_j)$ is $0$ since $A_0$ does not appear. This contradicts \eqref{equation_norms_Kummer}, therefore $[K(y_1,y_2)(z):K(y_1,y_1)]=p$.

By Lemma~\ref{lem_Sn}, $p=2$ and thus the characteristic is $\neq 2$. As $L/K(y_1,y_2)$ is an $S_{n-2}$-extension, it has a unique subextension of degree $2$ which is the fixed field of $A_{n-2} = A_n\cap S_{n-2}$ hence is generated by $\sqrt{\disc(f)}$. But $z$ also generates a quadratic subextension, hence
\begin{equation}\label{kummer2y1y2}
\frac{D(y_1)}{D(y_2)} \disc(f)  = z^2\disc(f)\in (K(y_1,y_2)^{\times})^2.
\end{equation}
Apply the norm map \eqref{nfirstnorm} to obtain
\begin{equation} \label{equation_norms_Kummer2} \frac{D(y_1)^n}{D(y_1) \cdots D(y_n)} \cdot \disc^{n-1}(f) \in (K(y_1)^{\times})^2.
\end{equation}
If $n$ is even, then $(1,\ldots, 1)\in V_{n-1}$ (as $p=2$). Therefore, by Sub-step 1a, $V\neq V_{1}, V_{n-1}, V_n$, that is, $V=0$, in contradiction to the assumption that $(1,-1,0,\ldots,0)\in V$. Thus, $n$ is odd. By \eqref{equation_norms_Kummer2} and \eqref{dthroughf}
\begin{equation}\label{kummer4y1y2}
D(y_1) \equiv \prod_{j=1}^{d} f(\alpha_j) \mod (F(y_1,A_0,A_2,\cdots, A_m)^{\times})^2.
\end{equation}
As $D$ is not a square and each of the $f(\alpha_j)$ is linear in $A_0$, and by the fact that $f(\alpha_j)$ and $f(\alpha_k)$ are associate only if $ \alpha_j=\alpha_k$, we must have that $A_0$ appears in the left hand side, which is a contradiction.

\vspace{5pt}
\noindent \textbf{Step 2:} General $r$.

Put $w_{i,j} = D_{i}(y_j)$ and as before let $V$ be the space of linear dependencies:
\begin{equation}\label{eq:MultiKummer}
V=\{ (v_{i,j})_{1 \le i \le r, 1\le j \le n} \in \FF_p^{nr} : \prod_{i=1}^{r}\prod_{j=1}^{n} D_i^{v_{i,j}}(y_j) \in (L^{\times})^p\}
\end{equation}
and we want to prove that  $V=0$.

Here the action of $S_n=\Gal(L/K)$ on the $w_{i,j}$ is by permuting the $j$-th index, so $V$ is an $S_n$-invariant space with respect of the action  of $S_n$ given by permuting  the columns.

Assume in contradiction that $V \neq 0$. We begin by constructing a matrix $B$ in $V$ of rank $1$. Let $A \in V$ be a non-zero matrix. Denote its columns by $v_1,\cdots,v_n$. If $\mathrm{rk}(A) =1$ we take $B=A$. Otherwise, assume without loss of generality that $v_1$ and $v_2$ are linearly independent. Then the matrix
\begin{equation*}
B = A - \left( v_2 \mid v_1 \mid v_3 \mid v_4 \mid \cdots \right) = \left( v_1-v_2 \mid v_2-v_1 \mid 0 \mid \cdots \right) \in V.
\end{equation*}
has rank $1$, as needed.

There are non-zero vectors $\mathbf{a}=(a_1,\ldots,a_r) \in \FF^r_p$, $\mathbf{b}=(b_1,\ldots,b_n) \in \FF^n_{p}$ such that
\begin{equation*}
B_{i,j} = a_i \cdot b_j.
\end{equation*}
The relation $B \in V$ is equivalent to
\begin{equation}\label{eq:manyrkummer}
\prod_{i=1}^{r} \prod_{j=1}^{n} D_{i}^{a_i \cdot b_j}(y_j) \in (L^{\times})^p.
\end{equation}
Let $D(T):=\prod_{i=1}^{r} D_i^{a_i}(T)$. We have $\prod_{j=1}^{n} D^{b_j}(y_j) \in (L^{\times})^p$, which by Step~1 implies that $b_j = 0$ for all $j$, contradicting the fact that $\mathbf{b} \neq 0$.  This concludes the proof of Case A.

\vspace{5pt}
\noindent \textbf{Case B: $\mathrm{char}(F)=p$}

From now on, $\wp(x)=x^p-x$. Let $v$ be the discrete valuation at the infinite prime; that is to say,  $v(h(T)/g(T)) = \deg h-\deg g$. Since $E/\FF_q(T)$ is tamely ramified at $v$, the extension generated by a root of $X^p-X=D_i(T)$ is unramified at $v$. This implies that there exists some $g_i(T)$ with $v(D_i+g_i^p-g_i)\geq 0$. We may replace $D_i$ by $D_i+g_i^p-g$, to assume without loss of generality that $v(D_i(T))\geq 0$. We may further assume that the roots in the denominators of $D_i(T)$ have multiplicity indivisible by $p$ \cite[\S2]{madden1978}.
The proof goes analogously to Case A: 

\vspace{5pt}
\noindent \textbf{Step 1:} $r=1$.

Put $D = D_1$ and $w_j= D(y_j)$. Let $V$ be the space of linear dependencies of the $w_j$-s:
\begin{equation}\label{V_def_Artin}
V= \{(v_1,\ldots, v_n) \in \FF_p^n :  \sum_{i=1}^{n} w_i v_i \equiv 0 \mod \wp(L)\}
\end{equation}
We need to prove that $V=0$. Since $V$ is an invariant subspace of $\FF_p^n$ under the action of $S_n=\Gal(L/K)$,  by Lemma~\ref{lem_invsub}  it suffices to prove that $V\neq V_1,V_{n-1}, V_n$.

\vspace{5pt}
\noindent \textbf{Sub-step 1a:} $(1,\ldots, 1) \not\in V$; hence $V\neq V_1,V_{n}$.

	We assume in contradiction that $(1,\ldots,1) \in V$.  In other words, there exists $z\in L$ such that
	\begin{equation}\label{eq:zp-z=sumDi}
	D(y_1) + \ldots + D(y_n) = z^p-z.
	\end{equation}
	Write $D=\frac{r_1(T)}{r_2(T)}$ where $r_1$, $r_2$ are coprime, and $r_2(T) = c\prod_{j=1}^{d} (T-\alpha_j)$ ($c \in F^{\times}, \alpha_j \in F$). By Lemma \ref{LemSym},
	\begin{equation}\label{eq:ArtinIden}
	z^p-z =\sum_{i=1}^{n} D(y_i) = \frac{h(\mathbf{A})}{\prod_{j=1}^{d} f(\alpha_j)},
	\end{equation}
	where $h(\mathbf{A}) \in F[\mathbf{A}]$ is coprime to the denominator as polynomials in $A_0$. In particular, $K(z)/K$ is a Galois subextension of the $S_n$-extension $L/K$ of degree $p$ or $1$. Put $H=\Gal(L/K(z))$ so that by Lemma~\ref{lem_Sn},  either $H=1$, $H=S_n$ or $H=A_n$, where the latter is possible only if $p=2$.
	
	If $H = 1$, then $K(z) = L$, so $n! = 1$ or $n! = p$, which contradicts  $n > 2$. Thus, $H\neq 1$.
	
	Now we show that $H\neq S_n$. If $H=S_n$, then $[K(z):K]=1$. Therefore $z \in K$. The denominator of $z^p-z$ is a (possibly trivial) $p$-th power in $K$, so that by \eqref{eq:ArtinIden} it follows $\prod_{j=1}^{d} f(\alpha_j)$ is a $p$-th power in $K$. Each $f(\alpha_j)$, as a polynomial in $A_0$, is linear with leading coefficient 1, so it must appear a multiple of $p$ times. On the other hand, by comparing the coefficient of $A_1$, the equality $f(\alpha_j)=f(\alpha_k)$ implies that $\alpha_j=\alpha_k$. As $\alpha_i$ have multiplicity indivisible by $p$ by our assumption on $r_2(T)$, we arrive to contradiction.

	So $H = A_n$ and $p=2$. In characteristic 2, we must use the Berlekamp discriminant $\mathrm{Berl}(f)$\footnote{See \cite{berlekamp1976} or \cite{carmon2015} for a recent use in a similar setting.} in place of the usual $\disc(f)$. There is a unique field $K'$ such that $K \subseteq K' \subseteq L$ with $\Gal(L/K') = A_n$, namely $K' = K(\delta)$ for $\delta$ which satisfies $\delta^2-\delta=\mathrm{Berl}(f)$. Thus, $ K(z) = K(\delta)$, or equivalently 
	\begin{equation}
	\wp(z) \equiv \wp(\delta) \bmod \wp(K),
	\end{equation}
	which by \eqref{eq:ArtinIden} becomes 
	\begin{equation}\label{eq:BerlCont}
	\mathrm{Berl}(f) \equiv \frac{h(\mathbf{A})}{\prod_{j=1}^{d} f(\alpha_j)} \bmod \wp(K).
	\end{equation}
	As in Sub-step 1a in the Kummer case, the linear-in-$A_0$ polynomial $f(\alpha_j)$ is coprime to $\disc(f)$. The Berlekamp discriminant is of the form $N(f)/\disc(f)$ for some polynomial $N$ in the coefficients of $f$. Thus, the denominators of the two fractions in \eqref{eq:BerlCont} are coprime as polynomials in $A_0$. By Lemma \ref{Lem:CopDenom}, this implies that the denominator $\prod_{j=1}^{d} f(\alpha_j)$ is a square, contradicting the fact that the $\alpha_j$ have odd multiplicity by our assumption on $r_2(T)$.

\vspace{5pt}
\noindent \textbf{Sub-step 1b:} $(1,-1 , 0, \ldots , 0) \notin V$; hence $V \neq V_{n-1}$ and $V \neq V_n$.

Assume in contradiction that $(1,-1,0,...,0) \in V$. So, there exists $z \in L$ such that
\begin{equation}\label{dy_1=dy_2^p_Artin}
D(y_1)=D(y_2) +z^p-z.
\end{equation}
We have $\Gal(L/K(y_1,y_2)) = S_{n-2}$ and  $K(y_1,y_2)(z)/K(y_1,y_2)$ Galois  of degree $1$ or $p$. Assume in contradiction that  $[K(y_1,y_2)(z):K(y_1,y_2)]=1$, then $z\in K(y_1,y_2)$. Applying the trace map
\begin{equation}\label{nfirsttrace}
T\colon K(y_1,y_2)\to K(y_1).
\end{equation}
on \eqref{dy_1=dy_2^p_Artin}, adding by $D(y_1)$, and considering \eqref{eq:ArtinIden}, we obtain
\begin{equation} \label{equation_traces_Artin} nD(y_1)=  \frac{h(\mathbf{A})}{\prod_{j=1}^{d} f(\alpha_j)} + T(z)^p-T(z) \equiv \frac{h(\mathbf{A})}{\prod_{j=1}^{d} f(\alpha_j)} \bmod \wp(K(y_1)).
\end{equation}
The field $K(y_1)$ is the field of rational functions in $A_0,A_2,\ldots ,A_m,y_1$ over $F$ since $A_1 = -\frac{A_0 + A_2y_1^2 + \cdots }{y_1}$. 

This implies that $f(\alpha_j)$ and $f(\alpha_k)$, as elements in $F(A_2,\cdots,A_m,y_1)[A_0]$ are associate if and only if $\alpha_j=\alpha_k$. The denominator in the left hand side of \eqref{equation_traces_Artin} does not involve $A_0$ and so it is coprime to the denominator in the right hand side. By Lemma \ref{Lem:CopDenom}, this means that the denominator in the right hand side is a $p$-th power, contradicting the fact that for every $j$, the multiplicity of $f(\alpha_j)$ is $ \not \equiv 0\mod p$. Therefore $[K(y_1,y_2)(z):K(y_1,y_1)]=p$, and by Lemma \ref{lem_Sn}, $p=2$. 

As $L/K(y_1,y_2)$ is an $S_{n-2}$-extension, it has a unique subextension of degree $2$ which is the fixed field of $A_{n-2} = A_n\cap S_{n-2}$ hence is generated by $\delta \in L$ which satisfies $\delta^2-\delta = \mathrm{Berl}(f)$. But $z$ also generates a quadratic subextension, hence
\begin{equation}\label{artin2y1y2}
D(y_1)-D(y_2) + \mathrm{Berl}(f)   = z^2-z+\mathrm{Berl}(f)\in \wp(K(y_1,y_2)).
\end{equation}
Apply the trace map \eqref{nfirsttrace} to obtain
\begin{equation} \label{equation_traces_Artin2} nD(y_1) - (D(y_1)+\ldots+D(y_n)) +(n-1) \mathrm{Berl}(f) \in \wp(K(y_1)).
\end{equation}
If $n$ is even, then $(1,\ldots, 1)\in V_{n-1}$ (as $p=2$). Therefore, by Sub-step 1a, $V\neq V_{1}, V_{n-1}, V_n$, that is, $V=0$, in contradiction to the assumption that $(1,-1,0,\ldots,0)\in V$. Thus, $n$ is odd. By \eqref{equation_traces_Artin2} and \eqref{eq:ArtinIden}
\begin{equation}\label{artin4y1y2}
\frac{r_1(y_1)}{r_2(y_1)} \equiv \frac{h(\mathbf{A})}{\prod_{j=1}^{d} f(\alpha_j)} \mod \wp(F(y_1,A_0,A_2,\cdots, A_m)).
\end{equation}
The denominator in the left hand side does not involve $A_0$ and so it is coprime to the denominator in the right hand side as a polynomial in $A_0$. By Lemma \ref{Lem:CopDenom}, this means that the denominator in the right hand side is a square. This contradicts the fact that for every $j$, the multiplicity of $f(\alpha_j)$ is odd, and that $f(\alpha_j)$ and $f(\alpha_k)$ are associate if and only if $\alpha_j = \alpha_k$. 
	
\vspace{5pt}
\noindent \textbf{Step 2:} General $r$.

Put $w_{i,j} = D_{i}(y_j)$ and let $V$ be the space of linear dependencies:
\begin{equation}\label{eq:MultiArtin}
V=\{ (v_{i,j})_{1 \le i \le r, 1\le j \le n} \in \FF_p^{nr} : \sum_{i=1}^{r}\sum_{j=1}^{n} v_{i,j} D_i(y_j) \in \wp(L)\}
\end{equation}
and we want to prove that  $V=0$. Assume in contradiction that $V \neq 0$. As in Step 2 of the Kummer case, there exists $B$ in $V$ of rank $1$. There are non-zero vectors $\mathbf{a}=(a_1,\ldots,a_r) \in \FF^r_p$, $\mathbf{b}=(b_1,\ldots,b_n) \in \FF^n_{p}$ such that
\begin{equation*}
B_{i,j} = a_i \cdot b_j.
\end{equation*}
The relation $B \in V$ is equivalent to
\begin{equation}\label{eq:manyrartin}
\sum_{i=1}^{r} \sum_{j=1}^{n} a_i b_j D_{i}(y_j) \in \wp(L).
\end{equation}
Let $D(T):=\sum_{i=1}^{r} a_i D_i(T)$. We have $\sum_{j=1}^{n} b_j D(y_j) \in \wp(L)$, which by Step~1 implies that $b_j = 0$ for all $j$, contradicting the fact that $\mathbf{b} \neq 0$.  This concludes the proof of Case B. \qed

\section{Proof of Theorem~\ref{thm:maintechnical}}\label{sec:pfthm:maintechnical}
First we prove the assertion for $\psi = \delta_\lambda$, where $\lambda$ is a $G$-factorization type supported on $e=1$ with $\deg \lambda = n$ and $\delta_{\lambda}\in \Lambda^*$ is given by
\[
\delta_\lambda(\lambda') = \begin{cases}
1,&\mbox{if }\lambda'=\lambda,\\
0,&\mbox{otherwise.}
\end{cases}
\]
Let $C =\{h \in G \wr S_n : \lambda_h = \lambda\}  \subseteq G\wr S_n$. Since $\lambda$ is supported on $e=1$ and $\deg \lambda=n$, by Lemma~\ref{Lem_lambda_conj}, $C$ is a conjugacy class. 

Write $f_0=T^n + \sum_{i=0}^{n-1}a_iT^i$ and put $W=W_{a,m}$ as in \eqref{eq:Wam}. Every $(w_0,\ldots, w_{n-1})\in W(\FF_q)$ corresponds in a bijective way to $f = T^n +\sum_{i=0}^{n-1} w_i T^i$ with $\deg(f-f_0)\leq m$.

By the explicit Chebotarev theorem\footnote{For a version which is sufficiently uniform in the parameters see either \cite[Appendix]{ABR} or \cite[Thm.~3]{entin2018}. In the former there is a mistake in the formulation of the theorem, in the notation of loc.\ cit.\ the error term should be dependent on the complexity of $S$ and not on the complexity of $R$ and $\deg \mathcal{F}$ as written. In the setting where $R$ is a polynomial ring in several variables and $S$ is the ring generated by adding roots of a polynomial $\mathcal{F}\in R[X]$, then the complexity of $S$ is bounded in terms of the complexity of $R$ and the total degree of $\mathcal{F}$.} and by Proposition~\ref{general_computation}, 
\[
\frac{\#\{ w\in W(\FF_q) : w \mbox{ is unramified in $U$ and $\phi_w \in C$}\}}{q^{m+1}} =   \frac{|C|}{|G\wr S_n|} + O(q^{-1/2}),
\]
where the implied constant is bounded in terms of the complexity of $U$, which is bounded in terms of $|G|$, $\genus(C)$, and $n$ and hence in terms of $B$.
This finishes the proof of this case since $\frac{|C|}{|G\wr S_n|} = \left<\delta_{\lambda}(\lambda_{(\xi,\sigma)})\right>_{(\xi,\sigma)\in C}$, for unramified $w$ we have $\delta_{\lambda}(f) = \delta_{\lambda}(\phi_w)$ by Lemma~\ref{prop:frob}, there are $O_B(q^{m})$ ramified $w$ (the zeros of $D$ given in \eqref{def:D}), and so
\[
\left<\delta_{\lambda}(f)\right>_{\deg(f-f_0)\leq m}=\frac{\#\{ w\in W(\FF_q) : w \mbox{ is unramified in $U$ and $\phi_w \in C$}\}}{q^{m+1}}  + O_B(q^{-1}).
\]

For general $\psi$, we partition $\Lambda = \Lambda_1 \cup \Lambda_2\cup \Lambda_3$, where $\Lambda_1$ consists of $\lambda$-s of degree $n$ supported on $e=1$, $\Lambda_2$ consists of the other $\lambda$-s of degree $n$, and $\Lambda_3$ consists of $\lambda$-s of degree $\neq n$.
This decomposes  $\psi = \psi_1 + \psi_2 +\psi_3$ with $\psi_i$ supported on $\Lambda_i$.
Now, as $\deg(f-f_0)\leq m$ implies that $\deg f=n$, we have
\[
\left< \psi_3(f) \right>_{\deg(f-f_0)\leq m} = 0.
\]
Since $\psi_2(f)=0$ if $f$ is squarefree and there are $O_B(q^m)$ non-squarefrees satisfying $\deg(f-f_0)\leq m$ \cite[Thm.~1.3]{keating2016}, we get that
\[
\left<\psi_2(f) \right>_{\deg(f-f_0)\leq m} = O_B(q^{-1}).
\]
The function $\psi_1$ decomposes as $\psi_1 = \sum_{\lambda\in \Lambda_1} \psi_1(\lambda) \delta_{\lambda_1}$, so by the special case proved above
\[
\begin{split}
\left<\psi_1(f)\right>_{\deg(f-f_0)\leq m } &= \sum_{\lambda\in \Lambda_1}\psi_1(\lambda)\left<\delta_\lambda(f)\right>_{\deg(f-f_0)\leq m } \\
&= \sum_{\lambda\in \Lambda_1}\psi_1(\lambda)\left<\delta_\lambda(\xi,\sigma)\right>_{(\xi,\sigma)\in G\wr S_n } + O_B(q^{-1/2})\\
&=\left<\psi_1(\xi,\sigma)\right>_{(\xi,\sigma)\in G\wr S_n } + O_B(q^{-1/2}).
\end{split}
\]
This completes the proof as $\psi_1(\xi,\sigma) = \psi(\xi,\sigma)$. \qed

\section{Non-geometric extensions}
Here we explain how the results for non-geometric extensions may be reduced to  geometric extensions over a field extension:
Let $G$ be a finite group, let $E/\FF_q(T)$ be a $G$-extension and let $E^{ab}$ be the fixed field of the commutator of $G$ in $E$. Assume that $E^{ab}$ is tamely ramified at infinity and that $E$ (or equivalently $E^{ab}$) is not geometric. 
Let $\FF_{q^\nu}$ be the algebraic closure of $\FF_q$ in $E$, $C_\nu=\Gal(\FF_{q^\nu}/\FF_q)$, and $H = \Gal(E/\FF_{q^\nu}(T))$.  By replacing $E$ with $E\FF_{q^\mu}$ for some large $\mu$, we may assume without loss of generality that $G=H\rtimes C_{\nu}$. Now we apply the construction of \S\ref{GT} to the extension $E/\FF_{q^\nu}(T)$ to get that the corresponding Galois group is $\Gal(M/K)=H\wr S_n$ and we have a diagram of fields whose right column is as in Figure~\ref{figure1}
\[
\xymatrix{
	&M\ar@{-}[d]\\
L_0=K_0(y_1,\ldots, y_n) \ar@{-}[r]
	&L=K(y_1,\ldots, y_n)\ar@{-}[d]\\
K_0=\FF_q(A_0,\ldots, A_m)\ar@{-}[r]\ar@{-}[u]
	&K=\FF_{q^\nu}(A_0,\ldots, A_m)
}
\]
Now, $\Gal(L_0/K_0)=S_n$ for the same reason that $\Gal(L/K)=S_n$, and $\Gal(K/K_0) \cong \Gal(\FF_{q^\nu}/\FF_q)=C_{\nu}$. Thus, $\Gal(L/K_0) = S_n\times C_\nu$.  With a little more effort, one can verify that in fact $M/K_0$ is Galois, and that $\Gal(M/K_0) = (H\wr S_n) \rtimes C_{\nu}$ with $C_{\nu}$ acting trivially on $S_n$ and acting diagonally on $H^n$. 
Now we can apply the higher dimensional Chebotarev theorem, to get a Chebotarev theorem in short intervals for this extension.

\appendix

\section{Norms in full intervals}\label{app:1}
We use the notation of \S\ref{sec:apps}. The goal of this appendix is to compute the mean value of $b_{E/\FF_{q}(T)}$ and of $r_{E/\FF_q(T)}$ in the most general setting of the limit $q^n\to \infty$.
\begin{theorem}\label{thm:bfullint}
Let $E/\FF_q(T)$ be a non-trivial geometric $G$-extension. Then
\begin{equation}\label{eqSWL}
\left< b_{E/\FF_q(T)}(f) \right>_{f\in M_{n,q}} = K_E\binom{n+\frac{1}{|G|}-1}{n} (1+ O_{\genus(E),|G|}(\frac{1}{\sqrt{q}n})),
\end{equation}
where $K_E$ is a positive constant that satisfies
\begin{equation}\label{eq:scalarestimate}
K_E = 1+O_{\genus(E),|G|}(\frac{1}{\sqrt{q}}).
\end{equation}
\end{theorem}

The mean value of $r_{E/\FF_q(T)}$ depends on  the following zeta function
\[
\zeta_{\mathcal{O}_E}(s) = \sum_{0 \neq I \text{ ideal in }\mathcal{O}_E} \frac{1}{(\# \mathcal{O}_E / I)^s}.
\]
\begin{proposition}\label{prop:rfullint}
Let $E/\FF_q(T)$ be a non-trivial geometric extension of degree $d$. If $n \gg_{\genus(E),d} 1$ then
\begin{equation}\label{eqSWLb}
\left< r_{E/\FF_q(T)}(f) \right>_{f\in M_{n,q}} = \lambda_E \log q,
\end{equation}
where $\lambda_E>0$ is the residue of $\zeta_{\mathcal{O}_E}(s)$ at $s=1$, and it satisfies
\begin{equation}\label{eq:lambdaestimate}
\lambda_E \log q = 1+O_{\genus(E),d}(\frac{1}{\sqrt{q}}).
\end{equation}
\end{proposition}

Unlike the rest of the paper, here the methods are analytic. 

\subsection{Bounds on prime counting functions}
We use the notation $\mathcal{P}_{n,q}$ introduced in \S\ref{sec:ChebFunc}, with one modification -- we do not include the infinite prime in $\mathcal{P}_{1,q}$. For any positive integer $f$, define
\begin{equation*}
\pi_{E;f}(n) = \sum_{\substack{P \in \mathcal{P}_{n,q},\\ f(P;E) = f}} 1
\end{equation*}
and set
\begin{equation*}
\psi_E(n) = \sum_{d  f \mid n} d  f  \pi_{E;f}(d).
\end{equation*}
\begin{lemma}\label{primebound}
We have
\[
|\pi_{E;1}(n) -\frac{q^n}{n|G|} | \ll \frac{\max \{\genus(E),|G|\}}{n} q^{n/2}.
\]
\end{lemma}
\begin{proof}
Let $e$ be the identity element of $G$. The number $\pi_{C;q}(n;E)$ with $C=\{ e\}$ is equal to $\pi_{E;1}(n)$, up to a contribution of ramified primes:
\begin{equation}\label{eq:piuptoram}
|\pi_{\{e\};q}(n;E) - \pi_{E;1}(n)| \le \sum_{P \in \mathcal{P}_{n,q}, \text{ ramified in }E} 1.
\end{equation}
The Riemann-Hurwitz formula \cite[Thm.~7.16]{rosen2002} shows that
\begin{equation}\label{eq:riemannhurconseq}
\sum_{P \in \mathcal{P}_{n,q}, \text{ ramified in }E} 1 \ll \frac{\max\{\genus(E),|G|\}}{n}.
\end{equation}
By \eqref{eq:Cheb_FF} with $C=\{e\}$, we have
\begin{equation}\label{eq:chebapp}
|\pi_{\{e\};q}(n;E)- \frac{q^n}{n|G|}|  \ll \frac{\max\{ \genus(E),|G|\}}{|G|} \frac{q^{n/2}}{n}.
\end{equation}
From \eqref{eq:piuptoram}--\eqref{eq:chebapp} and the triangle inequality, the proof follows.
\end{proof}
\begin{proposition}\label{proppsieconst}
We have
\[
\left| \psi_E(n) - \frac{q^{n}}{|G|} \right| \ll \max\{ \genus(E),|G| \} q^{\frac{n}{2}}.
\]
\end{proposition}
\begin{proof}
We separate the summands in $\psi_E(n)$ according to whether $d=n$ (a case which contributes $n\pi_{E;1}(n)$) or not:
\begin{equation*}
\psi_E(n) = n\pi_{E;1}(n) +T(n).
\end{equation*}
The triangle inequality gives us
\begin{equation}\label{triang}
\left|\psi_E(n) - \frac{q^n}{|G|}\right| \le \left|n\pi_{E;1}(n)-\frac{q^n}{|G|}\right| + T(n).
\end{equation}
We may bound $T(n)$ from above as follows, using the fact that $|\mathcal{P}_{n,q}| \le \frac{q^{n}}{n}$:
\[
T(n) \le \sum_{d=1}^{\lfloor \frac{n}{2} \rfloor}  \sum_{f=1}^{\lfloor \frac{n}{d} \rfloor} df \pi_{E;f}(d) \le n\sum_{d=1}^{\lfloor \frac{n}{2} \rfloor}  \sum_{f=1}^{\lfloor \frac{n}{d} \rfloor} \pi_{E;f}(d)\le n \sum_{d=1}^{\lfloor \frac{n}{2} \rfloor}  \frac{q^d}{d} .
\]
One can show by induction that
\[
\sum_{i=1}^{m} \frac{q^i}{i} \le 4\frac{q^m}{m}
\]
for all $m \ge 1$, which implies that 
\begin{equation}\label{s2upp22}
T(n) \ll q^{n/2}. 
\end{equation}
From \eqref{triang}, \eqref{s2upp22} and Lemma \ref{primebound}, we conclude the proof of the proposition.
\end{proof}
\subsection{Proof of Theorem~\ref{thm:bfullint}}
Consider the power series
\begin{equation}
D_E(u) = \sum_{f \in \FF_q[T], \text{ monic}} b_{E/\FF_q(T)}(f) u^{\deg f} = \sum_{n \ge 0 } \langle b_{E/\FF_q(T)}(f) \rangle_{f \in M_{n,q}} (qu)^{n}.
\end{equation}
Lemma \ref{lem:bcriteria} shows that $D_E$ admits the following Euler product:
\begin{equation*}
\begin{split}
D_{E}(u) &= \prod_{P \in \mathcal{P}_q} \left( 1+u^{\deg (P^{f(P;E)})}+u^{\deg (P^{2 \cdot f(P;E)})}+\ldots \right)\\
&= \prod_{P \in \mathcal{P}_q} \left(1-u^{f(P;E)\deg P}\right)^{-1}=  \exp\left( \sum_{P \in \mathcal{P}_q} \sum_{k \ge 1} \frac{u^{f(P;E)  \deg P \cdot k}}{k} \right).
\end{split}
\end{equation*}
From the definition of $\pi_{E;f}(n)$ and $\psi_E(n)$, we may write the above expression as
\begin{equation}\label{eq:deviapsie}
D_{E}(u) = \exp\left( \sum_{n \ge 1} \frac{u^n}{n}  \sum_{d f \mid n} d f  \pi_{E;f}(d) \right) = \exp\left( \sum_{n \ge 1} \frac{u^n}{n} \psi_E(n)\right).
\end{equation}
For any positive integer $n$, define
\begin{equation*}
e_{n} = \psi_E(n) - \frac{q^{n}}{|G|}.
\end{equation*}
Let
\begin{equation*}
a(u) = \exp\left( \sum_{n \ge 1} \frac{e_{n} u^{n}}{n} \right), \quad b(u) = (1-q u)^{-\frac{1}{|G|}} .
\end{equation*}
By \eqref{eq:deviapsie}, we have $D_{E}(u) = \exp\left(\sum_{n \ge 1} \frac{q^n u^n}{n|G|}  \right)  \exp\left( \sum_{n \ge 1} \frac{e_{n} u^n}{n} \right) = a(u) b(u)$, and so
\begin{equation*}
\langle b_{E/\FF_q(T)}(f) \rangle_{f \in M_{n,q}} = q^{-n}[u^{n}]D_E(u) = q^{-n}[u^n] a(u) b(u),
\end{equation*}
where $[u^n]F(u)$ is a notation for the coefficient of $u^n$ in a power series $F$. From Proposition \ref{proppsieconst} we have
\begin{equation*}
\left| e_n \right|  \ll \max\{ \genus(E), |G| \} q^{n/2}.
\end{equation*}
Hence we may apply \cite[Thm.~3.3]{gorodetsky2017} with $a(u), b(u)$ and obtain
\begin{equation*}
\langle b_{E/\FF_q(T)}(f )\rangle_{f \in M_{n,q}} = \binom{n+\frac{1}{|G|}-1}{n} \left( a(q^{-1}) + E \right),
\end{equation*}
where
\begin{equation*}
\left| E \right| \ll_{\genus(E),|G|} \frac{1}{\sqrt{q} n}.
\end{equation*}
which establishes \eqref{eqSWL} with $K_E = a(q^{-1})$. By \cite[Rem.~3.6]{gorodetsky2017} we have \eqref{eq:scalarestimate}. \qed

\subsection{Proof of Proposition~\ref{prop:rfullint}}

Let
\begin{align*}
\mathcal{Z}_{E}(u)&=\prod_{\mathcal{P} \text{ a prime in }E} (1-u^{\deg \mathcal{P}})^{-1},\\
\mathcal{Z}_{\mathcal{O}_E}(u)&=\prod_{\mathcal{P} \text{ a prime in }\mathcal{O}_E} (1-u^{\deg \mathcal{P}})^{-1},
\end{align*}
be the Dedekind zeta function of $E$ and of $\mathcal{O}_E$; in particular, $\mathcal{Z}_{\mathcal{O}_E}(q^{-s}) = \zeta_{\mathcal{O}_E}(s)$. They are related by 
\begin{equation}\label{eq:RelZetas}
\mathcal{Z}_{\mathcal{O}_E}(u)=\mathcal{Z}_E(u)   \prod_{\mathcal{P} \mid P_\infty} (1-u^{\deg \mathcal{P}}),
\end{equation}
where $P_\infty$ denotes the infinite prime of $\FF_q(T)$.
Suppose that $P_\infty =\mathcal{P}_1^{e_1}\cdots \mathcal{P}_m^{e_m}$ with $\mathcal{P}_i$ distinct primes of $E$ and put $f_i = f(\mathcal{P}_i;E)$. Then we have
\begin{equation}
\prod_{\mathcal{P} \mid P_{\infty}} (1-u^{\deg \mathcal{P}}) = \prod_{i=1}^{m}(1-u^{ f_i}).
\end{equation}
As $E/\FF_q(T)$ is geometric, the Riemann Hypothesis for Function Fields (RH) implies that $\mathcal{Z}_E$ is a rational function of the form
\begin{equation}\label{eq:rhforze}
\mathcal{Z}_{E}(u)=\frac{P_E(u)}{(1-qu)(1-u)},
\end{equation}
where $\deg P_E=2 \genus(E)$, $P_E(0)=1$ and the inverse absolute value of the roots of $P_E$ is $\sqrt{q}$. From \eqref{eq:RelZetas}--\eqref{eq:rhforze},
\begin{equation}\label{eq:zoeform}
\mathcal{Z}_{\mathcal{O}_E}(u)=\frac{P_E(u)}{(1-qu)} \cdot \frac{\prod_{i=1}^{m}(1-u^{ f_i})}{1-u} = \frac{\tilde{P}_E(u)}{(1-qu)},
\end{equation}
with $\tilde{P}_E(u) = P_E(u) \cdot \frac{\prod_{i=1}^{m}(1-u^{ f_i})}{1-u}$.
The function $\mathcal{Z}_{\mathcal{O}_E}(u)$ is a generating function for the mean value of $r_{E/\FF_q(T)}$:
\begin{equation}\label{eq:zoer}
\mathcal{Z}_{\mathcal{O}_E}(u)= \sum_{n \ge 0} \langle r_{E/\FF_q(T)}(f) \rangle_{f \in M_{n,q}} (qu)^n.
\end{equation}
From \eqref{eq:zoeform}, \eqref{eq:zoer}, it follows that if $n \ge \deg \tilde{P}_E $, we have
\[
\langle r_{E/\FF_q(T)}(f) \rangle_{f \in M_{n,q}} = \tilde{P}_E(\frac{1}{q}).
\]
As $\mathcal{Z}_{\mathcal{O}_E}(q^{-s}) = \zeta_{\mathcal{O}_E}(s)$, we have  $\tilde{P}_E(\frac{1}{q}) = \lambda_E \log q$. As $f_i, m \le d$, we have $\deg \tilde{P}_E = 2\genus(E) +\sum_{i=1}^{m} f_i-1 \ll_{\genus(E),d} 1$. Finally, RH implies that 
\[
\lambda_E \log q=  (1+O(\frac{1}{\sqrt{q}}))^{2 \genus(E)}  (1+O(\frac{1}{q}))^{O(d^2)} = 1+O_{\genus(E),d}(\frac{1}{\sqrt{q}}),
\]
as needed. \qed

\section*{Acknowledgment}
We wish to thank to Kumar Murty for helpful discussion on Chebotarev theorem.

\bibliographystyle{amsplain}
\bibliography{references}

\end{document}